\documentclass[10pt,fullpage]{article}
\title{On representation-finite algebras and beyond}
\usepackage{amssymb}
\usepackage{amsmath}
\usepackage{amsthm} 
\usepackage{enumerate}
\usepackage{pgf}
\usepackage{tikz}
\usetikzlibrary{arrows}
\usepackage[utf8]{inputenc}
\author{Klaus Bongartz}
\usepackage{amssymb}
\usepackage{amsmath}
\usepackage{amsthm} 
\usepackage[utf8]{inputenc}

\newtheorem{theorem}{Theorem}
\newtheorem{lemma}{Lemma}

\newtheorem{proposition}{Proposition}

\begin{document}
\title{On representation-finite algebras and beyond}
\author{Klaus Bongartz\\Bergische Universität Wuppertal\\Fachbereich C\\ Gaussstrasse 20\\D- 42219 Wuppertal\\Germany\\E-mail:bongartz@math.uni-wuppertal.de}
\date{}
\maketitle
\begin{abstract}
 We give a survey on the theory of representation-finite and certain minimal representation-infinite algebras. The main goals are the existence of multiplicative bases and of coverings 
with good properties. Both are attained using ray-categories. As applications we include a proof of a sharper version of the second Brauer-Thrall conjecture and of the fact that
 there are no gaps in the lengths of the indecomposables.
\end{abstract}

\section{Introduction}
At the conference ICRA 2012 in Bielefeld I gave a talk about the following result:

\begin{theorem}\label{gaps}
 Let $A$ be an associative algebra over an algebraically closed field. Then there is no gap in the lengths of the indecomposable $A$-modules of finite dimension.
\end{theorem}

The  proof of this elementary statement uses a generalization of an old result in covering theory and  many other facts  about  mild algebras i.e. algebras
 such that each proper quotient is representation-finite. A lot of mathematicians contributed to these results  which are not as well-known as Gabriels characterization of 
representation-finite quivers \cite{Dynkingab} and also not as easy to prove. Thus the organizers asked me to write a survey article about this material which  is
to some extent already discussed in the reviews  of Gabriel \cite{Revgab} and of Riedtmann \cite{Revriedt}.

The cornerstone is  the  article 'Representation-finite algebras and multiplicative bases' \cite{BGRS} which is beautiful but 
 hard to read. However, based on \cite{Standard,mild,Fischbacher,Geissm} one can now  generalize some of the main results and also simplify their proofs. I try to explain this with some 
details in 
section 3 after having introduced in section 2 the central concept of \cite{BGRS}: the ray-category attached to a distributive category. 

Each  mild ray-category has a  universal covering
 with good properties. This is  explained in section 4, before we summarize in the next section the results on coverings of the Auslander-Reiten quiver and on Galois coverings and relate 
them to the coverings of ray-categories. Section 6 contains two classification results.

 In the last section we give some interesting consequences of  the theory presented before. These include  a proof of theorem \ref{gaps}
and a sharper version of Brauer-Thrall 2, a criterion for finite representation type and finally some remarks on the classification of the representation-finite selfinjective algebras.

Most of the time we give no proofs but only precise references and comments that should help to understand a difficult argumentation. However, full proofs are given for the most important 
applications like theorem 1, Brauer-Thrall 2, the finiteness criterion and  representation-finite selfinjective algebras  or in situations where the existing proofs are simplified e.g. in 
sections 3.5, 3.8 and 5.3.

\section{Distributive categories and their  associated ray-categories}

\subsection{Some definitions, conventions and notations}We will always work over an algebraically 
closed field $k$ and we will consider only basic algebras so that the length of a module is just the dimension from now on.

Since we are using covering theory we consider often a locally bounded category $A$. This is a $k$-linear category such that different objects in $A$
 are not isomorphic, the endomorphism algebras $A(x,x)$
 are all local and for any $x$ the sum of the dimensions of all $A(x,y)$ and $A(y,x)$ are finite. As in the case of an algebra such a category has a uniquely determined locally-finite
 quiver $Q_{A}$ and there is a presentation of $A$, i.e. a surjective $k$-algebra homomorphism $\phi:kQ_{A} \rightarrow A$ such that the kernel $I$ is contained in the ideal generated by 
all paths of length $2$. However, $\phi$ and the kernel depend on some choices. We will always assume that $Q_{A}$ is connected, whence countable. 

 An $A$-module $M$ is just a covariant $k$-linear functor from $A$ to 
the category of $k$ vector spaces.  This can also be seen as a representation of $Q_{A}$ that satisfies the relations
 imposed by $I$. We denote the category of all modules by $A$-Mod and the full subcategory of modules with $dim \,M:= \sum_{x \in A} dim \, M(x) < \infty$ by $A$-mod.
A full subcategory $B$ of $A$ is called convex if for any $x,y$ in $B$ also all $z$ lying on a path from $x$ to $y$ in $Q_{A}$ belong to $B$. The quiver of a convex subcategory $B$  of $A$ 
is a subquiver of $Q_{A}$ and any $B$-module can be extended by zeros to an $A$-module. Using this we  often consider $B$-modules as $A$-modules. A line of length $d$ is a full convex subcategory given by a quiver of Dynkin type $A_{d}$ without relations. $A$ is Schurian if we have $dim A(x,y) \leq 1$ 
for all $x$ and $y$. If the quiver of $A$ contains no oriented cycle we call $A$ directed.

By definition $A$ is locally representation-finite if for each $x$ there are up to isomorphism only finitely many indecomposable modules $U$ in $A$-mod with $U(x)\neq 0$.
$A$ is mild if any proper quotient $A/I$ is locally representation-finite and $A$ is minimal representation-infinite if it is 
mild but not locally representation-finite.

If a locally bounded category $A$ is finite
 i.e. if it has only finitely many objects one can take the
 direct sum $B$ of all $A(x,y)$ 
and endow it with the multiplication induced by the composition in $A$. That way the finite locally bounded categories correspond to basic finite dimensional algebras and locally 
representation-finite reduces to representation-finite and so on.

A base category $B$ is a category with uniquely determined zero morphisms in each $B(x,y)$ such that different objects are not isomorphic, all endomorphisms different 
from the identity are nilpotent and for each $x$ there are only finitely many morphisms different from zero starting or ending in $x$. The linearization $kB$ of such a base
 category is the
 $k$-linear category with the
 same objects as $B$ and $kB(x,y)$ is the quotient of the $k$-vector space with basis $B(x,y)$ divided by the subspace spanned by the zero morphism. The composition is 
induced by the composition in $B$. A $B$-module is just a $kB$-module, $B$ is mild if $kB$ is so, $B-mod$ is $kB-mod$ and so on.

Any base category $B$ is given by a quiver $Q_{B}$ and relations. Again we always assume that $Q_{B}$ is connected. The points of $Q_{B}$ are the objects of $B$ and the arrows the irreducible morphisms in $B$, i.e. the non-zero 
non-invertible morphisms that are not a composition of non-invertible morphisms. The  ( non-linear ) path category $PQ_{B}$ having as morphisms the paths and additional zero-morphisms 
admits an obvious 
full functor to $B$ which is the identity on the objects and the arrows. Thus each base category can be
described by a quiver and a stable equivalence relation on its paths. Here an equivalence relation is stable provided it is invariant under left and right multiplication. 

We illustrate these lengthy definitions by the following example occurring in Riedtmanns work. We look at the quiver $Q$ shown in figure 1 and the relations
 $R_{1}=\{\beta \alpha  \sim  \gamma^{2}, \alpha\beta  \sim0\}$ and $R_{2}=\{\beta\alpha \sim \gamma^{2},\alpha\beta \sim  \alpha\gamma\beta,\gamma^{4} \sim0\}$.
 For $i=1,2$ let $B_{i}$ be the quotient 
of the non-linearized path category by the stable equivalence relation generated by $R_{i}$. Then these two base categories are  not isomorphic. However, their linearizations $kB_{i}$ are 
isomorphic 
iff the characteristic of $k$ is different from 2.

\setlength{\unitlength}{1cm}
\begin{picture}(8,2.5)
\put(5.1,0.9){\vector(1,0){0.7}}
\put(5.9,1.1){\vector(-1,0){0.7}}
\put(6.4,1.0){\circle{0.8}}
\put(5,1){ \circle*{0.2}}
\put(6,1){\circle*{0.2}}
\put(7,1){$\gamma$}

\put(5.5,1.3){$\alpha$}
\put(5.5,0.5){$\beta$}
\put(6,1.2){\vector(0,-1){0.1}}
\put(5,0){figure 1}

\end{picture}

\subsection{Distributive categories}

A locally bounded category is called distributive provided its lattice of two-sided ideals is distributive. The following two observations of Jans and Kupisch are basic.
\begin{proposition}\label{JansKupisch}
 Let $A$ be a locally bounded category. 
\begin{enumerate}
 \item ( \cite[Corollary 1.3]{Jans} ) The lattice of two-sided ideals in $A$ is finite iff $A$ is finite and distributive.
\item (  \cite[Satz 1.1]{Kupisch1} ) $A$ is distributive iff all $A(x,x)$ are isomorphic to a quotient $k[T]/(T^{m(x)})$ for some natural number $m(x)$ and all $A(x,y)$ are cyclic over $A(x,x)$ 
from the right or over $A(y,y)$ from the left.
\end{enumerate}

\end{proposition}
 It is clear that  quotients and full subcategories of distributive categories are again distributive. A directed category is distributive
iff it is Schurian. Both algebras occurring in the example of Riedtmann are distributive.

Jans introduces in his paper the famous Brauer-Thrall conjectures 1 and 2 for a finite dimensional algebra A. Recall that Brauer-Thrall 1 - abbreviated BT 1 -  says that
 the representation-finite algebras are 
the only ones with a bound for the dimensions of the indecomposables, whereas Brauer-Thrall 2 says that for a representation-infinite algebra $A$ there are infinitely many dimensions
in which  one finds infinitely many isomorphism classes of indecomposable $A$-modules. BT 1 was proven by Roiter in 1968 and there is the long article \cite{NR} of Nazarova and 
Roiter aiming at a proof of BT 2, but the first complete proof  is 
due to Bautista in 1984. It uses many of the results surveyed in this article and I will give 
later a  proof of a more precise version of BT 2 that was proposed by myself at the Luminy-conference in 1982 for a certain type of algebras which includes in fact all minimal representation-infinite 
algebras as we know now.

We say that $A$ satisfies BT 0 if there is no gap in the lengths of the indecomposable $A$-modules. The reason why this was not
 formulated as a conjecture could be 
that the truth of BT 1 and BT 2 
for group
algebras in characteristic $p$ follows directly from the special case of $p$-groups by using Krull-Remak-Schmidt and induction resp. restriction to a $p$-Sylow-subgroup \cite{Higman}.
 However BT 0 is not obvious for
 group algebras because the  behaviour of simple and indecomposable modules under induction and restriction is difficult to control. 

Jans proves BT 2 in his paper \cite[theorem2.1]{Jans} for algebras that are not distributive. The truth of BT 0 for these algebras is apparently shown for the first time in \cite[section 1]{Gaps}
 published  2009 in the archive. 
Shortly after that a stronger statement 
was proven by Ringel. Namely he defines recursively on the dimension the notion of an accessible module by requiring that all simples are accessible and that a module 
of dimension $d>1$ is
 accessible provided 
it is indecomposable and it admits an accessible submodule or quotient of dimension $d-1$.

\begin{theorem} \label{Ringel}( \cite{Ringel2} )
 If $A$ is not distributive it has an accessible module in each dimension.
\end{theorem}

So it remains to prove theorem 1 only for distributive categories. We will see that also in this case there is an accessible module in each dimension where an indecomposable exists at all.

\subsection{Ray-categories and stem-categories}
Throughout this section $A$ denotes a distributive category. By Kupischs result the subbimodules of $A(x,y)$ are linearly ordered with respect to inclusion and they are cyclic over $A(x,x)$ or $A(y,y)$.
To  a distributive category  one can attach  two base  categories both having the same objects as $A$: the stem-category $\hat{A}$ introduced by Kupisch 1965 in \cite[Satz 1.2]{Kupisch1} and the ray-category
$\vec{A}$ introduced 1985 in \cite{BGRS}.

The set of morphisms $\hat{A}(x,y)$ is the finite set of subbimodules of $A(x,y)$ and  the composition is defined by the product of 
subbimodules. Of course, this composition is associative and  the non-zero morphisms in $\hat{A}$ produce in the linearization 
$k\hat{A}$ a filtered multiplicative basis. This means that each non-zero product of two base vectors is again a base vector and that the basis contains all identities and bases for 
powers of the radical. Kupisch proves in \cite[Satz1.2]{Kupisch1} that $A$ and the linearization $k\hat{A}$ have 'the same' ideal
 lattice and he uses the stem-category in his analysis of the structure of representation-finite symmetric algebras ( \cite{Kupisch1,Kupisch2} ) culminating in the proof
 that these algebras and more generally 
representation-finite selfinjective 
algebras have a filtered  multiplicative
basis ( \cite{Kupisch3,Kupisch4} ).

The definition of the ray-category $\vec{A}$ is more complicated. For later use we introduce some more definitions and notations. Given a
 morphism $\mu$ in $A(x,y)$ we denote by $A_{\mu}$ the subbimodule generated by $\mu$. If $A_{\mu}$ is cyclic over $A(x,x)$ from the right resp. over
$A(y,y)$ from the left resp. from both sides  we call $\mu$  cotransit resp. transit resp. bitransit. The orbit of $\mu$ 
 under the obvious action of the groups of units in
 $A(x,x)$ and $A(y,y)$ is the ray $\vec{\mu}$ and $\vec{A}(x,y)$ is the set of these rays. The map $\mu \mapsto A_{\mu}$ induces a bijection  between the sets $\vec{A}(x,y)$ and 
$\hat{A}(x,y)$. The composition inside $\vec{A}$ is defined in the next proposition where also some properties of $\vec{A}$ are listed. We include a proof of some non-trivial parts as 
an illustration.

\begin{proposition}\label{ray}( \cite[1.7,1.13]{BGRS} )
 Let $A$ be a distributive category and let $\mu:x \rightarrow y$ and $\nu:y \rightarrow z$ be two morphisms in $A$.  Then the following is true:
\begin{enumerate}
 \item Either  $\nu \mu$ belongs to $\nu \,rad\,A(y,y)\,\mu$ or not. In the first case there are morphisms $\mu'$ and $\nu'$ in the rays $\vec{\mu}$ and $\vec{\nu}$ with $\nu'\mu'=0$ and 
we define $\vec{\nu}\vec{\mu}=0$. 
In the second case the compositions $\nu'\mu'$ of all morphisms in the rays $\vec{\nu}$ and $\vec{\mu}$ form the non-zero ray $\overrightarrow{\nu\mu}$ which we take as $\vec{\nu}\vec{\mu}$.
\item  The composition just defined is associative.
\item If $I$ is a two-sided  ideal  in $A$ then $\vec{I} =\{\vec{\mu} | \mu \, \in \,I \}$ is an 'ideal' in $\vec{A}$ with $\vec{A}/\vec{I} \simeq \overrightarrow{A/I}$. The map $I \mapsto \vec{I}$
 is an isomorphism of the ideal lattices.
\item For each $x$ the endomorphism set $\vec{A}(x,x)$ is isomorphic to the semigroup $H_{n}=\{1,\alpha,\alpha^{2},\ldots ,\alpha^{n}=0\}$ with $n+1$ elements, where $n$ depends on $x$. 
Similarly, for all $x,y$ the set $\vec{A}(x,y)$ is cyclic under the action of $\vec{A}(x,x)$ from the right or under the action of $\vec{A}(y,y)$ from the left.
\item From $\vec{\nu} \vec{\mu}=\vec{\nu}\vec{\lambda}\vec{\mu }\neq 0$ it follows that $\vec{\lambda}$ is an identity. Similarly, $\vec{\nu}_{1}=\vec{\nu}_{2}$ follows from
$\vec{\mu}\vec{\nu}_{1}=\vec{\mu}\vec{\nu}_{2}\neq 0$ or from $\vec{\nu}_{1}\vec{\mu}=\vec{\nu}_{2}\vec{\mu}\neq 0$. 
 
\end{enumerate}

\end{proposition}

\begin{proof} a) If $\nu\mu \in \nu\,rad A(y,y)\,\mu$ we have $\nu\mu=\nu\rho\mu$ for some nilpotent $\rho$, whence $\nu (id_{y}-\rho )\mu = 0 = \nu' \mu' $  with $\nu' = \nu$ and
 $\mu '= (id_{y}- \rho ) \mu $. 

 If $\nu\mu \notin \nu\,rad A(y,y)\,\mu$  take any  $\nu' \in \vec{\nu}$ and 
$\mu' \in \vec{\mu}$. We have to show that $A_{\nu\mu}=A_{\nu'\mu'}$. Assume $A_{\nu ' \mu '}$ is properly contained in $ A_{\nu \mu }$. Up to duality $A(x,z)$ is cyclic over $A(x,x)$ and we get $\nu' \mu' = \nu \mu \rho$ for some nilpotent $\rho$. Furthermore we have $\nu'= \alpha \nu \beta$ 
and $\mu'=\gamma \mu \delta$ for some invertible $\alpha ,\beta ,\gamma ,\delta$. This leads to $\alpha \nu \beta \gamma \mu \delta= \nu \mu \rho$ and 
$\alpha \nu \beta \gamma \mu = \nu \beta \gamma \mu \zeta$ for some invertible $\zeta$. If $\mu$ is cotransit we get $\nu \mu \rho= \nu\mu \eta \zeta\delta$ for some invertible $\eta$ whence 
the contradiction $\nu \mu \in \nu \mu \, rad A(x,x)$. Thus $\mu$ is transit and we obtain $\nu \theta \mu= \nu \kappa \mu$ with some nilpotent $\theta$ and some 
 $\kappa =  a id_{y} + \lambda$ where $\lambda$ is nilpotent and $a$ a non-zero scalar. This implies the contradiction  $\nu\mu \in \nu\,rad A(y,y)\,\mu$. Thus we have 
$A_{\nu\mu} \subseteq A_{\nu'\mu'}$ because the subbimodules are linearly ordered by inclusion. For  $\nu'\mu' \in \nu' \,rad A(y,y) \,\mu'$ we get $\nu'' \in \vec{\nu}$ and
 $\mu'' \in \vec{\mu}$ with $0=\nu'' \mu''$ by the first case, whence 
$A_{\nu''\mu''}$ is properly contained in $A_{\nu\mu}$ which is impossible as just shown. Thus we have $\nu'\mu' \notin \nu' \,rad A(y,y) \,\mu'$. Interchanging the roles of the dashed and 
undashed letters we obtain the wanted equality.

b) From a) one gets that $\vec{\nu }( \vec{\mu } \vec{\lambda }) \neq 0$ iff $\nu' (\mu '
\lambda ' )\neq 0$ holds for all $\nu ' \in \vec{\nu}$, $\mu' \in \vec{\mu}$ and $\lambda' \in \vec{\lambda}$ in which case all these products belong to the same ray.
 The analogous statement holds for 
$(\vec{\nu } \vec{\mu }) \vec{\lambda } \neq 0$ and the associativity follows from the associativity of the composition in $A$.

e) The first cancellation law is an immediate consequence of part a). The other two laws follow from the first using part d).
\end{proof}

Observe that the map $\mu \mapsto A_{\mu}$ induces an equivalence  between the categories $\vec{A}$ and
$\hat{A}$ iff the case $\nu \mu \in \nu \,rad\,A(y,y)\,\mu \neq 0$ never occurs.

A base category $C$ satisfying the last two properties listed in the proposition is called an abstract ray-category. It is easy to see that the ray-category attached to 
the  ( distributive ) category
 $kC$ is $C$ again, so that each abstract ray-category comes from a distributive category. In sharp contrast the original distributive category $A$ is in general not equivalent
 to $k\vec{A}$ which we define as the standard form $A^s$ of $A$. We call $A$ standard if we have $A\simeq A^s$. For instance we have in Riedtmanns example $\overrightarrow{kB_{2}} \simeq B_{1}$.

Now the obvious question is why the ray-categories should be better than the stem-categories which are much easier to define. As long as one is 'only' interested in the existence of 
 multiplicative bases the distinction is not  so important. But if one wants to get  universal coverings  with good properties one has to consider ray-categories. This is crucial for 
the proofs 
of BT 0 and BT 2. The different behaviour with respect to coverings can already
 be seen at Riedtmanns example. Whereas $B_{1}$ has a nice universal covering without oriented cycles $B_{2}$ is its own universal covering. We will make this more precise in section 4.

Let $A$ be a distributive category. For a path $v=\alpha_{n}\alpha_{n-1}\ldots \alpha_{2}\alpha_{1}$ in $Q_{A}$ of length $l(v)=n$  we set $\vec{v}=\vec{\alpha}_{n}\vec{\alpha}_{n-1}
\ldots \vec{\alpha}_{2}\vec{\alpha}_{1}$. 
We call $v$ stable 
if $\vec{v}\neq 0$ and a minimal zero-path if $\vec{v}=0$,
 but $\vec{\alpha}_{n}\vec{\alpha}_{n-1}\ldots \vec{\alpha}_{2}\neq 0 \neq \vec{\alpha}_{n-1}\ldots \vec{\alpha}_{2}\vec{\alpha}_{1}$. A non-zero morphism $\mu \,\in \,\vec{A}(x,y)$ is deep 
resp. profound
if it is annihilated by the radicals of $\vec{A}(x,x)$ and $\vec{A}(y,y)$ resp. by all irreducible morphisms. A contour is a pair $(v,w)$ of paths with
$\vec{v}=\vec{w}\neq 0$. The contour $(w,v)$ is reverse to $(v,w)$.  Two paths $v$ and $w$ are interlaced if they are equivalent under the equivalence relation generated by the relation $\sim$ defined by $v \sim w$ iff $v=abc$ and 
$w=adc$ with $\vec{b}=\vec{d}\neq 0$ and $l(a) + l(c)>0$. A contour $(v,w)$ is essential if $v$ and $w$ are not interlaced and deep if $\vec{v}$ is deep. It is clear that the 'kernel' of the natural presentation
$PQ \rightarrow \vec{A}$ is the smallest stable equivalence relation such that each minimal zero-path is equivalent to $0$ and $v$ is equivalent to $w$ for each essential contour. Since 
the spaces $A(x,y)$ and $A^{s}(x,y)$ always have the same finite dimension we get that $A$ is
 isomorphic to $A^s$ iff there is a presentation $\phi$ such that all minimal zero-paths and  all the differences $v-w$ coming from an essential contour $(v,w)$
are annihilated by $\phi$.

\section{The structure of mild k-categories}
\subsection{The main results}
In this section we explain the main results from \cite{BGRS} and their proofs, but in their generalized and simplified versions made possible by \cite{Standard,mild,Fischbacher,Geissm}.

\begin{theorem}\label{multbasis}
 Let $A$ be a distributive category such that the ray-category $\vec{A}$ is mild. Then we have:
\begin{enumerate}
 \item $A$ is standard if the characteristic is not $2$.
\item $A$ has always a filtered multiplicative basis.
\end{enumerate}
\end{theorem}

In particular by this and the next theorem there are only finitely many isomorphism classes of representation-finite algebras in each dimension. This answers a question asked by Gabriel in \cite{Gabaus}.

\begin{theorem}\label{faithful} Let $A$ be a distributive category.
 \begin{enumerate}
\item $A$ is mild resp. locally representation-finite iff $A^s$ is so. In that case the Auslander-Reiten quivers  are isomorphic.
\item If $A$ is mild and if it has a faithful indecomposable it is standard. In particular a minimal representation-infinite algebra is standard.
\end{enumerate}
\end{theorem}

The central statement is part a) of theorem \ref{multbasis}.
To describe its proof let $Q$ be the common quiver of $A$ and $\vec{A}$. If $\phi: kQ \rightarrow A$ and $\phi': kQ \rightarrow A$ are
 two presentations of $A$,
the rays of $\phi(\alpha)$
and of $\phi'(\alpha)$ coincide and we call this ray $\vec{\alpha}$. One gets all presentations just by all choices of elements in the various rays corresponding to the arrows of
$Q$. As explained at the end of the last paragraph 
 we have to find a presentation $\phi$  that annihilates all minimal zero-paths and  all the differences $v-w$ coming from essential contours $(v,w)$. Starting from an arbitrary presentation we will achieve this  in three steps explained in the sections 3.3 to 3.5.

Step 1: The  minimal zero-paths have length 2 and two different ones have no arrow in common, so that by changing the choice for one arrow in each minimal zero-path one can annihilate all of them.

Step 2: There are only three types of non-deep essential contours allowed and they are pairwise disjoint. In $char \,k\neq 2$ there is for each such contour a new choice of one arrow
 occuring in
 the contour - but not in a minimal zero-path - 
 such that $v-w$ is annihilated. 

Step 3: For each arrow $\alpha$ we multiply the element $\phi(\alpha) \in \vec{\alpha}$ chosen before by an appropriate non-zero scalar to annihilate the differences $v -w$ for all contours.
The existence of these scalars is equivalent to the vanishing of a certain cohomology group.

The proof of step 3 given on three pages in \cite[8.3 - 8.6]{BGRS} is very elegant whereas the  proofs of the first two steps require a careful local analysis of $\vec{A}$ 
 affording some endurance which is only at the very end 
rewarded by the nice
 structural
 results one obtains. The main working tool for the proofs of the first two steps is described in the next section: the cleaving diagrams due to Bautista, Larri\'{o}n and Salmer\'{o}n \cite[section 3]{BGRS}.

\subsection{Cleaving diagrams}

A diagram $D$ in a ray-category $P$ is just a covariant functor $ F:D \rightarrow P$ from another ray-category $D$ to $P$ respecting zero-morphisms. Then $F$ is called 
cleaving  iff it satisfies 
the following two conditions and their duals: a) $F\mu =0$ iff $\mu=0$;
b) If $\alpha \epsilon D(x,y)$ is irreducible and $F\mu:Fx \rightarrow Fz$ factors through $F\alpha$ then $\mu$ factors already through $\alpha$. These conditions are very easy to verify.

For any diagram $F:D \rightarrow P$ the restriction $F^{\bullet}: P-Mod \rightarrow D-Mod$ has a left adjoint $F_{\bullet}$ and if $F$ is cleaving, any $M$ in $D-mod$ is a direct summand of 
$F^{\bullet}F_{\bullet}M$ \cite[sections 3.2,3.8]{BGRS}. It follows that $P$ is not locally representation-finite resp. satisfies BT 2 if $D$ does so.

 In this article $D$ will always be given by its
 quiver $Q_{D}$,
 that has no oriented cycles, and some relations. Two  paths between the same points give always the same morphism, and zero relations are written down explicitely.
 The diagram $F:D \rightarrow P$ is then defined by drawing the quiver of $D$ with relations and by writing the morphism $F\alpha$ in $P$ close to each arrow $\alpha$. 

For  example let $D$ be the ray-category with the natural numbers $0,1,\ldots $ as objects and with arrows $2n \rightarrow 2n+1$ and $2n+1 \leftarrow 2n+2$ for all $n$ or let $D_{k}$ be
 the full subcategory of $D$ supported by the natural numbers $\leq k$ where $k\geq 1$. Then a cleaving functor 
from $D$ resp $D_{k}$ to $P$ is called an infinite zigzag resp. a zigzag of length $k$.\vspace{0.5cm}

\setlength{\unitlength}{0.8cm}
\begin{picture}(20,2)

\put(1,2){\vector(1,-1){1}}

\put(3,2){\vector(-1,-1){1}}\put(0.9,1.5){$\rho_{1}$}

\put(3,2){\vector(1,-1){1}}\put(1.9,1.5){$\rho_{2}$}
\put(5,2){\vector(-1,-1){1}}\put(2.9,1.5){$\rho_{3}$}

\put(5,2){\vector(1,-1){1}}
\put(7,2){\vector(-1,-1){1}}

\put(7,2){\vector(1,-1){1}}\multiput(9,1.5)(0.5,0){10}{-}
\put(6,0){figure 2}
\end{picture}\vspace{0.5cm}  A functor from $D$ resp. $D_{k}$ to $P$
 is just an infinite resp. finite sequence of morphisms
 $(\rho_{1},\rho_{2},\ldots   )$ in $P$ such 
that $\rho_{2i}$ and $\rho_{2i+1}$ always have common domain and $\rho_{2i-1}$ and $\rho_{2i}$ common codomain.
 The  functor is cleaving iff none of the $\rho_{i}$  factors through one of its 'neighbored' morphisms.  The domain of $\rho_{1}$ is called the start of the zigzag.
A crown in $P$  is a zig-zag that becomes periodic after n steps, i.e. one has $\rho_{1}=\rho_{n+1}$ for some even $n\geq 4 $. A  zigzag is called low if none of the $\rho_{i}$ is a 
profound morphism. The ray-category $P$ is zigzag-finite resp. weakly zigzag-finite if in each point only finitely many zigzags resp. low zigzags start.
By Königs graph theorem this means that there  is no infinite zigzag in $P$ resp. no infinite low zigzag.  For instance any mild ray-category is weakly zigzag-finite.

In \cite[section 3]{BGRS} the cleaving functors are defined for $k$-linear categories and applied in that context. There the definition and the verification that a given
 functor is cleaving are
 much more difficult. Therefore it is important that the whole proof of theorem \ref{multbasis} can be done at the elementary combinatorial level whereas the transfer from $\vec{A}$ to $A$
is postponed to theorem \ref{faithful} which is also proved by elementary means. The possibility to proceed like that is
 already mentioned
 in \cite[section 3.8 c)]{BGRS}, but  there it would require some non-elementary results from \cite{Coverings,Bretscher} and the proof of theorem \ref{faithful} even depends on  BT 2.
  
 \vspace{0.5cm}.

\subsection{Zero relations and critical paths}
There are two types of minimal zero-paths $v=\alpha_{n}\alpha_{n-1}\ldots \alpha_{2}\alpha_{1}$ in $A$. Either $v$ is annihilated by each presentation $\phi$ or not in which case we call $v$
a critical path. 

\begin{theorem}\label{zero} Let $\vec{A}$ be mild. then we have:
\begin{enumerate}
 \item ( Structure theorem for critical paths ) Each critical path has length $2$.
\item ( Disjointness theorem for critical paths ) Two critical paths sharing an arrow are equal.
\end{enumerate}
\end{theorem}
 The proof of this theorem is a relatively easy and short application of the technique of cleaving diagrams. It is sketched in \cite[section 13.11]{Buch} and given with full details in
 \cite[section 4]{BGRS}. Note that for a critical path $\alpha_{2}\alpha_{1}$ the first ray $\vec{\alpha}_{1}$ is not cotransit, the second not transit.
 
Now for each critical path $v$ the initially given presentation can be changed at one arrow of the path to annihilate $v$ and this can be done for all critical paths at once because of the 
disjointness theorem. We end up with a presentation $\phi$ annihilating all zero-paths.

\subsection{Commutativity relations and non-deep contours}
For a contour $C$ we denote by $P(C)$ the full subcategory of $\vec{A}$ supported by the points occuring as starting or ending points of arrows in $v$ or $w$ and by $Q(C)$ the quiver of 
$P(C)$ which is in general not a subquiver of $Q$.

Figure 3 describes three  ( families of ) ray-categories by quiver and relations. Each of these contains a non-deep contour $(v,w)$ and $\vec{v}$ is always bitransit.
 For obvious reasons the contours $C$ as well as the  categories $P(C)$ are called  penny-farthings, dumb-bells and diamonds respectively.

\setlength{\unitlength}{1cm}
\begin{picture}(15,8)

\put(8,4){\circle*{0.1}}
\put(7,2){\circle*{0.1}}
\put(5,1){\circle*{0.1}}
\put(0,4){\circle*{0.1}}
\put(3,1){\circle*{0.1}}
\put(1,2){\circle*{0.1}}
\put(1,6){\circle*{0.1}}
\put(3,7){\circle*{0.1}}
\put(5,7){\circle*{0.1}}
\put(7,6){\circle*{0.1}}
\put(8.5,4){\circle{1}}
\put(7,6){\vector(1,-2){1}}
\put(8,4){\vector(-1,-2){1}}
\put(8.0,3.9){\vector(0,1){0.1}}
\put(9.2,4){$\rho$}
\put(5,7){\vector(2,-1){2}}
\put(7,2){\vector(-2,-1){2}}
\put(5,1){\vector(-2,0){2}}
\put(3,7){\vector(2,0){2}}
\put(3,1){\vector(-2,1){2}}
\put(1,6){\vector(2,1){2}}
\multiput(0,4)(0.1,0.2){10}{\circle*{0.01}}
\multiput(0,4)(0.1,-0.2){10}{\circle*{0.01}}

\put(7,5){$\alpha_{n}$}
\put(7,3){$\alpha_{1}$}
\put(5.8,3.9){$x=y=x_{0}$}
\put(6.5,2){$x_{1}$}

\put(5.8,6){$x_{n-1}$}
\put(8,3){$n\geq 1,\alpha_{1}\alpha_{n}=0$}
\put(8,2.5){$v=\alpha_{n} \ldots \alpha_{1}, w =\rho^{2}$}
\put(7.5,2){$ 0= \alpha_{e(i)}\ldots \alpha_{1}\rho\alpha_{n}\ldots \alpha_{i+1}$}
\put(7,1.5){$e:\{1, \ldots ,n-1\} \rightarrow \{1,\ldots ,n\}$}
\put(7,1){$e$ non-decreasing}

\end{picture}

\setlength{\unitlength}{0.5cm}
\begin{picture}(20,7)

\put(11,4){\circle*{0.2}}
\put(15,7){\circle*{0.2}}
\put(15,1){\circle*{0.2}}
\put(19,4){\circle*{0.2}}
\put(11,4.2){\vector(4,3){3.8}}
\put(15,6.8){\vector(-4,-3){3.8}}
\put(11,4){\vector(4,-3){3.8}}
\put(15,1){\vector(4,3){3.8}}
\put(18.8,4){\vector(-4,3){3.8}}
\put(15.2,7){\vector(4,-3){3.8}}

\put(10,4){$x$}
\put(19.5,4){$y$}
\put(15.5,7.2){$z$}
\put(15.5,0.7){$t$}

\put(12,5.5){$\gamma$}
\put(14,5.5){$\lambda$}
\put(16,5.5){$\kappa$}
\put(17.5,5.5){$\alpha$}

\put(12,2.5){$\delta$}
\put(17.5,2.5){$\beta$}

\put(11,0){$v=\beta\delta, w = \alpha \gamma$,}
\put(16,0){$\lambda \kappa =0,$}
\put(19,0){$\kappa \alpha= \gamma \lambda$}

\put(1,4){\circle*{0.1}}
\put(3,4){\circle*{0.1}}
\put(0.5,4){\circle{1}}
\put(3.5,4){\circle{1}}
\put(1,4){\vector(2,0){2}}
\put(1,4.1){\vector(0,-1){0.1}}
\put(3,4.1){\vector(0,-1){0.1}}
\put(0.5,3.9){$x$}
\put(3.3,3.9){$y$}
\put(1,4.5){$\lambda$}
\put(2,4.5){$\mu$}
\put(3,4.5){$\rho$}
\put(1,3){$v=\mu\lambda,w = \rho \mu$}
\put(1,2){$\lambda^{r} =0=\rho^{s}$}
\put(0,1){ $min \{r,s\} =3, max \{r,s\} \leq$ 5}
\put(9,-2){figure 3}
\end{picture}\vspace{2cm}

\begin{theorem}\label{structure} Let $A$ be a distributive category such that  $\vec{A}$  is mild. Then we have:
\begin{enumerate}
 \item ( Structure theorem for non-deep contours ) Any non-deep essential contour of $A$ is equal or reverse to a dumb-bell, a penny-farthing or a diamond.
\item ( Disjointness theorem for non-deep contours ) Two  essential non-deep contours  sharing  a point are equal or reverse.
\end{enumerate}

\end{theorem}

The proof of part a) given in \cite[sections 5 to 7]{BGRS} is complicated especially for the case of diamonds. There is a simpler proof in \cite[section 2]{mild} using the following 
obvious strategy. For a non-deep essential contour $C=(v,w)$ we choose paths $v=v_{1}\ldots v_{n}$ and $w=w_{1}\ldots w_{m}$ from $x$ to $y$. Up to duality it suffices to consider the case
 that  $\vec{v}$ is transit. We also choose a path $p=p_{1}p_{2}\ldots p_{r}$ such that $\rho=\vec{p}$ generates the radical of $\vec{A}(y,y)$ and we use the abbreviations
 $\alpha=\vec{v}_{1}$,$\beta=\vec{w}_{1}$,$\gamma=\vec{v}_{2}\ldots \vec{v}_{n}$ and  $\delta=\vec{w}_{2}\ldots \vec{w}_{m}$. Then the contour induces in $\vec{A}$  the 
 diagram shown in figure 4. 
\vspace{0.5cm}

\setlength{\unitlength}{0.5cm}
\begin{picture}(15,6)
\put(8,5){\vector(4,1){4}}
\put(8,5){\vector(3,-1){3}}
\put(12,6){\vector(1,-1){1}}
\put(11,4){\vector(2,1){2}}
\put(8,2){\vector(4,1){4}}
\put(8,2){\vector(3,-1){3}}
\put(12,3){\vector(1,-1){1}}
\put(11,1){\vector(2,1){2}}
\put(13,5){\vector(0,-1){3}}
\put(8,-1){figure 4}
\put(10,6){$\gamma$}
\put(12.3,5.8){$\alpha$}
\put(10,3){$\gamma$}
\put(12.3,2.8){$\alpha$}
\put(9,4){$\delta$}
\put(9,1){$\delta$}
\put(12,4){$\beta$}
\put(12,1){$\beta$}
\put(13.5,3.5){$\rho$}
\end{picture}

\vspace{1cm}

The proof of part a) is just a careful analysis of the fact that the obvious subdiagram of type $\tilde{D}_{5}$ or some diagrams deduced from it cannot be cleaving if $\vec{A}$  is mild. 
Part b) is proved in \cite{mild} and originally in \cite{BGRS} under the stronger assumption that the contours share an arrow.

By the definition of a non-deep contour $C=(v,w)$ and because $\vec{v}$ is  bitransit there is always an invertible morphism $\xi_{C}$ with $\phi(v)=\xi _{C}\phi(w)$ and in 
case of a penny-farthing $C$ we have $\xi_{C}= a^{2}_{C} id + b_{C}\phi(\rho_{C})$ because of $\phi(\rho_{C})^{4}=0$. To get the new presentation $\phi'$ wanted in step 2 we set
$\phi'(\rho_{C})=\xi_{C}\phi(\rho_{C})$ for each dumb-bell $C$, $\phi'(\alpha_{C})=\xi_{C}\phi(\alpha_{C})$ for each diamond and finally $\phi'(\rho_{C})=(a \, id + \frac{b}{2a}\phi(\rho_{C})
\phi(\rho_{C})$ for each penny-farthing. Here we use the fact $2$ is invertible in $k$. All these choices are independent of each other and we set $\phi'(\zeta)=\phi(\zeta)$ for the
 remaining arrows. Then all $v-w$ coming from non-deep essential contours are annihilated by $\phi'$ and so are all zero-paths since we have only changed arrows that are bitransit.

In a penny-farthing where  $\alpha_{1}\alpha_{n}$ is a zero-path one can  find  in all characteristics a new presentation annihilating $v-w$ by setting
 $\phi'(\alpha_{1})=\phi(\alpha_{1})\xi_{C}^{-1}$.
All critical paths are still annihilated because $\alpha_{1}$ does not lie on a critical path. Thus the only penny-farthings that cause serious trouble in characteristic $2$ are those where
$\alpha_{1}\alpha_{n}$ is a critical path.

\subsection{Contour functions and cohomology}

Let $A$ be a distributive category with ray-category $P=\vec{A}$. A contour function is a map $c$ that assigns to each contour $(v,w)$ of $A$ a non-zero scalar such that $c(u,w)=c(u,v)c(v,w)$ and 
$c(sv,sw)=c(v,w)=c(vt,wt)$ hold whenever this makes sense. A contour function is called exact if there is a function $e$ from the set of arrows of $Q_{P}$ to $k^{\ast}$ such that
$$c(v,w)=\delta(e) :=e(\alpha_{n})e(\alpha_{n-1}) \ldots e(\alpha_{1})(e(\beta_{m})e(\beta_{m-1}) \ldots e(\beta_{1}))^{-1} $$where $v$ and $w$ are the paths $\alpha_{n}\ldots \alpha_{1}$ and
 $\beta_{m}\ldots \beta_{1}$. 

The set $C(P)$ of all contour functions is an abelian group under pointwise multiplication and the set $E(P)$ of all exact contour functions is a subgroup.
As shown in \cite[section 8.1]{BGRS}, the quotient $H(P)$ is isomorphic to the second cohomolgy group $H^{2}(\vec{A},k^{\ast})$ defined in the next section, but this is irrelevant for us.

Now let $\phi$ be a presentation annihilating all zero paths and non-deep contours. Then there is for each contour $(v,w)$ a uniquely determined non-zero scalar $c(v,w)$ such that 
$\phi(v)=c(v,w)\phi(w)$ and this defines a contour function. Namely for a non-deep contour we have $c(v,w)=1$ and for a deep contour $\phi(v)$ and $\phi(w)$ are both generators of the one-dimensional socle of the bimodule 
$A(x,y)$. By the next theorem - called Roiters vanishing theorem in \cite{BGRS} - we have $c=\delta(e)$ for some function $e: Q_{1} \rightarrow k^{\ast}$ and then the new presentation $\phi'$ with $\phi'(\alpha)=e(\alpha )^{-1}\phi(\alpha)$  
induces the wanted isomorphism $k\vec{A} \simeq A$.

If $\vec{A}$ is mild and there are four arrows starting or ending in a point of the quiver then all compositions of irreducible morphisms vanish in $\vec{A}$. Thus we can assume in step 3 that at most 
three arrows start or end in a point.
\begin{theorem} \label{vanishing} If $P$ is a weakly zigzag-finite ray-category such that at most three arrows start or end at a point of its quiver then we have $H(P)=0$.
 
\end{theorem}

I give some details of the proof for two reasons: First, it is in my opinion the most ingenious single argument of the article on multiplicative bases, and second, the above statement is more general
 than the vanishing theorem proven in \cite{BGRS} only for zigzag-finite ray-categories. This generalization is due to Geiss who observed in its unpublished  'Diplomarbeit' that the
 proof of \cite{BGRS} still works for weakly zigzag-finite ray-categories. The reader should be warned that we have to change the definitions of \cite{BGRS} slightly to keep the arguments
 valid.

If the quiver of $P$ is finite and contains no oriented cycle, the proof of the vanishing theorem is easy and it is given in \cite{Zyk}. One considers a source in the quiver 
and proceeds by induction. This should  also work in the general situation, but - to cite A.V.Roiter - the question is: Induction on what?

Well, here are the definitions needed to create a kind of source or  sink in a quiver with oriented cycles. A tackle of length $n$ with start in $y$ is just a low zigzag
$z=(\rho_{1},\rho_{2}, \ldots ,\rho_{n})$ with start in $y$ whose last morphism $\rho_{n}$ is irreducible. The efficiency $e(z)$ of the tackle is the word 
$(d(\rho_{1}),d(\rho_{2}),\ldots ,d(\rho_{n}))$ and we order these words lexikographically. The tackle is efficient if its efficiency is maximal among the efficiencies of the
tackles starting 
in $y$. If $P$ is mild only finitely many tackles start in a fixed point so that there is always an efficient tackle as soon as there is one. The key lemma reads as follows:

\begin{lemma}\label{keylemma}
 Let $P$ be a ray-category and let $z=(\rho_{1},\rho_{2}, \ldots ,\rho_{n})$ be an efficient tackle starting in $y$. Choose paths $r_{i}$ with $\vec {r}_{i}=\rho_{i}$. Suppose that 
$(v=qr_{n}p, w=w_{m}w'w_{1})$ is an essential contour with paths $q,p,w'$ and arrows $w_{1},w_{m}$. Then $p$ resp. $q$ has length $0$ if $n$ is odd resp. even.
\end{lemma}

\begin{proof}We consider the case  $n=5$. This will make clear how to treat the general case. Suppose that $\vec{p}$ is not an identity. Then the sequence 
$z_{5}=(\vec{r}_{1},\ldots,\vec{r}_{4},\vec{q}\vec{r}_{5},\vec{w}_{m})$ contains no profound morphism and it cannot be a zigzag because then its efficiency is greater than that of $z$.
 Since $v$ and $w$ are not interlaced there exists a non-trivial path $q_{3}$  such that $\vec{q}_{3}\vec{r}_{4}=\vec{q}\vec{r}_{5}$.

Next consider the sequence $z_{4}=(\vec{r}_{1},\ldots,\vec{r}_{3},\vec{r}_{4}\vec{p},\vec{w}_{1})$ not containing a profound  morphism. Again this cannot be a tackle because then its efficiency is 
too big. Since $v_{4}=q_{3}r_{4}p$ is interlaced with $v$ it is not interlaced with $w$ and so there is a non-trivial path $p_{2}$ with $\vec{r}_{3}\vec{p}_{2}=\vec{r}_{4}\vec{p}$.

Similarly we look at $z_{3}=(\vec{r}_{1},\ldots,\vec{q}_{3}\vec{r}_{3},\vec{w}_{m})$ and  $z_{2}=(\vec{r}_{1},\vec{r}_{2}\vec{p}_{2},\vec{w}_{1})$ and we find non-trivial paths $q_{1}$ 
and $\vec{p}_{0}$ with $\vec{q}_{1}\vec{r}_{2}=\vec{q}_{3}\vec{r}_{3}$ and 
$\vec{r}_{1}\vec{p}_{0}=\vec{r}_{2}\vec{p}_{2}$ so that finally the sequence $z_{1}=(\vec{q}_{1}\vec{r}_{1},\vec{w}_{m})$ is a tackle with larger efficiency than $z$. This contradiction ends the proof of the lemma.
\end{proof}\vspace{2cm}

\setlength{\unitlength}{0.8cm}
\begin{picture}(20,5)

\put(1,5){\vector(1,-1){2}}
\put(1,4.3){$r_{1}$}\put(3.5,4.3){$r_{2}$}\put(6,4.3){$r_{3}$}\put(7.5,4.3){$r_{4}$}\put(9.5,4.3){$r_{5}$}\put(12.5,4.3){$w'$}

\put(12,6.5){$w_{1}$}\put(10.5,6){$p$}\put(7,6){$p_{2}$}\put(12,1.5){$w_{m}$}\put(4,6){$p_{0}$}

\put(10.6,2){$q$}\put(9.5,2){$q_{3}$}\put(4,2){$q_{1}$}

\put(5,5){\vector(-1,-1){2}}

\put(5,5){\vector(1,-1){2}}
\put(9,5){\vector(0,-1){1}}

\put(9,4){\vector(2,-3){2}}
\put(9,5){\vector(-1,-1){2}}
\put(11,7){\vector(-1,-1){2}}
\put(11,7){\line(-5,-1){5}}
\put(6,6){\line(-5,-1){5}}
\put(11,7){\vector(-3,-1){6}}
\put(11,7){\vector(1,-1){1}}
\put(12,6){\vector(0,-1){4}}
\put(7,3){\vector(2,-1){4}}
\put(3,3){\vector(4,-1){8}}
\put(12,2){\vector(-1,-1){1}}
\put(6,0){figure 5}
\end{picture}\vspace{0.5cm} 

As already said the foregoing lemma is just a slight generalization of \cite[8.4]{BGRS}. The point is that the domain of $w_{1}$ behaves like a source in the remainig part of the proof of 
the vanishing theorem. We do not repeat the arguments, but only explain why they are still correct.
The section 8.5 remains valid for a weakly zigzag-finite ray-category
 as is easy to see  and we look 
now at the proof of lemma 8.6 in \cite {BGRS} thereby finishing the proof of theorem \ref{vanishing}. The only thing to be modified is the start of the induction.
If there is no tackle in our new sense any arrow $\alpha$ starting in $y$ induces a profound morphism $\vec{\alpha}$ and we still can take $a=0$ in the proof of lemma 8.6.

\subsection{The neighborhoods of non-deep contours}

To go on with the proofs we need to know how a penny-farthing is related to the whole ray-category. This is analyzed in \cite{Standard} using some partial results from \cite{BGRS}.
A short complete proof is given in \cite[section 4.2]{mild}.

\begin{theorem}\label{MIPF} Let $C$ be a penny-farthing in a mild ray-category $P$. Suppose that $P(x_{0},y)\neq 0 $ for some $y$ not in $C$. Then we have $n=2$ and we are in one of the following three situations:
\begin{enumerate}
 \item There is an arrow $\beta:x_{0} \rightarrow b$ 
and this is the only arrow outside $C$ ending or starting in $C$. We have $\beta\rho=0$, $\delta\beta=0$ for all arrows $\delta$ starting in $b$ and $y=b$.
\item There is an arrow $\gamma:x_{1} \rightarrow c$ and this is the only arrow outside $C$ ending or starting in $C$. We have $\gamma\alpha_{1}\rho=0$, $\delta\gamma=0$ for all arrows starting in  $c$ and $y=c$.
\item There is an arrow $\beta:x_{0} \rightarrow b$ as well as an arrow $\gamma:x_{1}  \rightarrow c$. These are the only two arrows outside $C$ ending or
 starting in $C$. We have $0=\beta\rho$, $0=\beta\alpha_{2}$, $0=\delta\beta$ for all arrows starting in $b$,  $0=\gamma\alpha_{1}$, $0=\epsilon\gamma$ for all arrows starting in $c$ and $y=b\neq c$.
\end{enumerate}
In all three cases there are no additional arrows ending in $b$ or $c$.

\end{theorem}

Analogous results hold for dumb-bells and diamonds and this leads in \cite{mild} to the following  result:

\begin{theorem}\label{mininf} A minimal representation-infinite ray-category has no non-deep contour.
 
\end{theorem}

\subsection{The case of characteristic $2$}
We have already seen at Riedtmanns example that in characteristic $2$ the linearizations of the ray-category and the stem-category are not always isomorphic and we know that this trouble is caused by
 certain penny-farthings. So let $\mathcal{P}$ be the set of all penny-farthings in $A$ such that $\vec{\alpha_{1}}\vec{\rho}\vec{\alpha_{n}} \neq 0$. For any subset $\mathcal{N}$ of 
$\mathcal{P}$ we define a base category $\vec{A}_{\mathcal{N}}$ having the same objects and morphisms as $\vec{A}$. The composition of two morphisms $\vec{\nu}$ and $\vec{\mu}$ also coincides
with the composition in $\vec{A}$ except for the case where the domains and codomains all belong to the same penny-farthing contained in $\mathcal{N}$. In this case the composition is the ray 
$\vec{\lambda}$ such that $A_{\lambda}$ is the product of $A_{\nu}$ and $A_{\mu}$ in the stem-category. Thus $\vec{A}_{\mathcal{N}}$ is a mixture of ray- and stem-categories.

\begin{theorem}\label{char2} Suppose the characteristic of $k$ is $2$.
 Let $A$ be a distributive category such that $\vec{A}$ is mild. Then we have with the above notations:
\begin{enumerate}
 \item The composition defined above in $k\vec{A}_{\mathcal{N}}$ is associative. The category $k\vec{A}_{\mathcal{N}}$ is distributive with ray-category $\vec{A}$, but it is not standard provided 
${\mathcal N}$ is not empty. 
\item Each distributive category $B$ with $\vec{B} \simeq \vec{A}$ is isomorphic to $k\vec{A}_{\mathcal{N}}$ for some subset $\mathcal{N}$ of $\mathcal{P}$. Two such categories are
 isomorphic iff the subsets are conjugate under the automorphism group of $\vec{A}$.
\end{enumerate}

\end{theorem}

The proof of the first part uses some of the information contained in theorem  \ref{MIPF} and the proof of the second part runs parallel to the proof for characteristic different from 2. This is explained
 very   well in \cite[section 9]{BGRS}. One uses
 again the vanishing of a certain cohomology group.

\subsection{The proof of theorem \ref{faithful}}
We give some details because the proof of these results in \cite[section 9.7]{BGRS} is too complicated and the one in \cite[section 13.17]{Buch} contains a minor error.
 
Let $A$ be a distributive category such that $\vec{A}$ is mild. First we show that $A$ is standard if it has a faithful indecomposable. As the proof of theorem \ref{multbasis} shows this is
 clear if there is no penny-farthing $P$ as in figure 3 with $\vec{\alpha_{1}}\vec{\rho}\vec{\alpha_{n}} \neq 0$ in $\vec{A}$. So let $P$ be such a penny-farthing. If
 $A(y,x_{0})=0=A(x_{0},y)$
holds for all $y \notin P$ then $A(x_{0}, \,\,)$ is projective-injective, whence the only candidate for a faithful indecomposable, and annihilated by  $\alpha_{1}\rho \alpha_{n}$.
Thus, up to duality, we can assume $A(x_{0},y) \neq 0$ for some $y \notin P$. Then we are in one of the three situations described in theorem \ref{MIPF}. By a well-known result \cite{Martinez}
the points $b$ and $c$ can always
 be separated into a receiver and an emitter and the quiver of $A$ and $\vec{A}$ is separated into a connected component containing $x_{0}$ and at most two other components.
Since there is a faithful indecomposable these components are actually empty and we are in one of the three situations described in theorem \ref{MIPF}. As shown in \cite{Standard}
 by a direct calculation the Auslander-Reiten quivers of $A$ and $A^s$ coincide in all three cases and there is no faithful indecomposable.

Next, let $A$ be finite and representation-finite. Using the correspondence of proposition \ref{ray} between  ideals of $A$ and $\vec{A}$ we know by induction that $\vec{A}$ is mild. If it
 is minimal 
representaion-infinite it has a faithful indecomposable and so we have   $A\simeq k\vec{A}$ by the above, a contradiction. Thus $\vec{A}$ is also representation-finite.

Reversely, let $\vec{A}$ be finite and representation-finite. Again by induction $A$ is mild. If it is minimal representation-infinite it has a faithful indecomposable and we end 
up again with the contradiction that $A\simeq k\vec{A}$. 

We have shown for finite distributive categories that $A$ is representation-finite iff $\vec{A}$ is so. This implies easily that $A$ is locally representation-finite iff $\vec{A}$ is so.
The analogous statement for mildness follows from the correspondence between the ideals of $A$ and $\vec{A}$. Finally, the Auslander-Reiten quivers are isomorphic because for each 
penny-farthing one has either a projective-injective corresponding to the point $x_{0}$ or one is in one of the three cases from theorem \ref{MIPF}. In the second case the Auslander-Reiten quivers 
of $A$ and its standard-form are glued together by the same rules from the isomorphic Auslander-Reiten quivers obtained by separating all occurring points $b$ or $c$ into emitters and 
receivers. In the first case one knows how the Auslander-Reiten sequences with a projective-injective in the middle look like.

\section{The topology of a  ray-category}

\subsection{The simplicial complex and the universal covering of a base category}

The following material is from \cite[sections 1.10,10]{BGRS}. For each base category $B$ one defines a simplicial complex $S_{\bullet}B$ by taking the set of objects as $S_{0}B$ and the set of n-tuples $(\mu_{n},\ldots, \mu_{2},\mu_{1})$ of 
composable morphisms with $\mu_{n} \ldots \mu_{2}\mu_{1}\neq 0$ as $S_{n}B$. The face operators are defined in the usual way by dropping a morphism at the ends or by composing two in between 
and the degeneracy operators are defined by inserting identities ( see \cite[section 1.10]{BGRS} ). Let $C_{n}B$ be the free abelian group with basis $S_{n}B$ and define the differential
$d_{n}:C_{n}B \rightarrow C_{n-1}B$ by the alternating sum of the appropriate face operators, i.e. by $$d_{n}(\mu_{n},\ldots ,\mu_{2},\mu_{1})=(\mu_{n},\ldots ,\mu_{2})-
(\mu_{n},\ldots ,\mu_{2}\mu_{1})+ \ldots +
(-1)^{n}(\mu_{n-1},\ldots ,\mu_{2},\mu_{1}).$$ Then one obtains a chain complex $C_{\bullet}B$  whose homology groups are denoted by $H_{n}B$ whereas 
$H^{n}(B,Z)$  is the $n$-th cohomology group of $Hom( C_{\bullet}B,Z)$ for any abelian group $Z$.  

A functor $F:B' \rightarrow B$ between base categories is a covering if it satisfies the following conditions a), b) and the dual of b). Condition a) says that $F\mu=0$ is
 equivalent to $\mu=0$. Condition b) means that any point $x$ in $B$ can be lifted to a pont $x'$ and any $\mu:x \rightarrow y$ can be lifted to a unique $\mu':x' \rightarrow y'$.
 It follows that $\mu$ is irreducible iff $F\mu$ is, whence a covering induces a covering between the quivers. The covering $\pi:\tilde{B} \rightarrow B$ is 'the' universal covering if 
for any covering $F:B' \rightarrow B$ any $x$ in $B$ with liftings $x'$ in $B'$ and $\tilde{x}$ in $\tilde{B}$ there is exactly one functor $G:\tilde{B} \rightarrow B'$ with $\pi= F G$
 and $G\tilde{x}=x'$. Then $G$ is again a covering and  even an automorphism for $F=\pi$. The group of all these automorphisms is called the fundamental group $\Pi_{B}$ of $B$ and $B$
 is simply connected if this group is trivial.

It is easy to see that
one obtains the universal covering by the following construction. A walk $w=\alpha_{n}\ldots \alpha_{1}$ of length $n$ from $x$ to $y$  is a formal composition of arrows $\beta$ in $Q_{B}$ and formal
inverses $\beta^{-1}$ such that the domains and codomains fit together well and $x$ is the domain of $\alpha_{1}$, $y$ the codomain of $\alpha_{n}$. Two walks $v$, $w$ can be composed to the 
walk
$wv$ if the end of $v$ is the start of $w$. The homotopy is the smallest equivalence relation on the set of all walks such that:1) $\alpha\alpha^{-1} \sim id_{y}$ and
$ \alpha^{-1}\alpha \sim id_{x}$  for all $\alpha:x \rightarrow y$, 2) $v \simeq w$ and $v^{-1} \simeq w^{-1}$ for all paths $v,w$ mapped to the same non-zero morphism under the canonical
 presentation $PQ_{B} \rightarrow B$ and 3) $v \sim w$ implies $uv \sim uw$ resp. $vu \sim wu$ whenever these compositions are defined. Now the points of the universal covering are the homotopy 
classes of walks with a fixed start $x$ and the fundamental group consists of the homotopy classes with start and end in $x$. The multiplication is induced by the composition of walks. Since 
$Q_{B}$ is connected this construction is essentially independent of the chosen base point $x$.

The following 'Hurewicz-isomorphism' from \cite[10.4]{BGRS} is useful:
\begin{proposition}
 Let $P$ be a connected ray-category with fundamental group $\Pi$ and universal covering $\tilde{P}$. Then  $\Pi/[\Pi,\Pi]$ is isomorphic to $H_{1}P$. In particular one has $H_{1}\tilde{P}=0$.
\end{proposition}
The elementary definitions above are familiar from algebraic topology, but for an arbitrary base category the construction of the universal covering following these lines is an impossible 
task because it involves the word problem for groups as shown in \cite{delapena1}.  However, the reader can easily verify that in Riedtmanns example $B_{2}$ is simply connected 
whereas $B_{1}$ admits the universal covering shown in figure 6. Here the horizontal arrows are mapped onto the loop and the relations in the universal covering are the lifted ones.
\vspace{0.5cm}

\setlength{\unitlength}{0.8cm}
\begin{picture}(20,5)
\multiput(3,3)(2,0){6}{\circle*{0.2}}
\multiput(1,1)(2,0){7}{\circle*{0.2}}

\multiput(1,1)(2,0){6}{\vector(1,0){1.8}}
\multiput(1,1)(2,0){6}{\vector(1,1){1.8}}
\multiput(3,3)(2,0){5}{\vector(1,-1){1.8}}
\multiput(0,2)(0.3,0){3}{\circle*{0.1}}
\multiput(14,2)(0.3,0){3}{\circle*{0.1}}
\put(6,0){figure 6}
\end{picture}
\vspace{0.5cm}

\subsection{The main results for ray-categories} 

A ray-category $P$ is called interval-finite if the quiver $Q_{P}$ of $P$ is directed and if there are only finitely many paths between any two points of $P$. For a point $x$ 
in an interval-finite ray-category we denote by $P_{x}$ the set of all $y\neq x$ such that $P(x,y)\neq 0$. We order this set by $y \leq z$ iff the only non-zero morphism from $x$ to $z$ 
factors through $y$.
 Furthermore we denote by $Px$ the full subquiver of $Q_{P}$ consisting of the points $y$ where no path ending in $x$ starts. We say that $x$ is separating if each 
connected component of  $Px$ contains at most one connected component of the Hasse-diagram of $P_{x}$.

\begin{theorem}\label{separation}(\cite{Bretscher,Criterion} ) Let $P$ be a weakly zigzag-finite interval-finite ray-category. 
\begin{enumerate}
 \item $H_{1}P=0$ holds iff $H_{1}C=0$ holds for all finite convex subcategories $C$ of $P$. In that case all objects are separating.
\item If $P$ is finite and all objects are separating then $H_{1}P=0$ holds.
\end{enumerate}

\end{theorem}
The next result \cite{Fischbacher} of Fischbacher is of central importance. 
\begin{theorem}\label{Fischbacher}
Let $\pi: \tilde{P} \rightarrow P$ be the universal covering of a zigzag-finite ray-category $P$. Then we have:
\begin{enumerate}
 \item The fundamental group $\Pi_{P}$ is free.
\item $H^{2}(P,Z)=0$ for all abelian groups $Z$.
\item $\tilde{P}$ is interval-finite.
\end{enumerate}

\end{theorem} 
 By this result and section 5 one can always use covering theory instead
 of cleaving diagrams in the proof of the multiplicative basis theorem, but that does not abbreviate the arguments.

Fischbachers  proof of a) and b) is by an induction based on his 'reduction-lemma' which shows for a ray-category $P$ with at least one contour  - among other things - the existence of some 
arrows such that $P$ and the quotient of $P$ by the ideal generated by these arrows have the same fundamental group. The proof of the reduction lemma  uses only the key lemma \ref{keylemma} 
 so that parts a) and b) remain valid 
if $P$ is  only weakly zigzag-finite. 
The proof of part c) is based on theorem \ref{separation} and it does no longer work in the weakly zigzag-finite case.

 Nevertheless, we have the following result from \cite{Gaps} which plays an essential role in the proofs of BT 0 and of BT 2 in our sharper version.

\begin{theorem}\label{Fischbong} Let $P$ be a mild ray-category. Then the statements  a),b) and c) of theorem \ref{Fischbacher} are true.
 
\end{theorem}
For the proof we can assume that $P$ is minimal representation-infinite and that it contains an infinite zigzag $Z$ and a profound morphism that is not irreducible. Any profound
 morphism $\mu$ occurs in $Z$ infinitely many times 
because otherwise $P/\mu$ 
still contains the  infinite zigzag consisting of the end of $Z$ where $\mu$ no longer occurs. But if a zigzag contains three times the same morphism,  one can construct a crown $C$ in the 
obvious way. Thus $P$ is finite. Now one proves the following crucial result \cite{Gaps}.

\begin{proposition}\label{long}
 Let $P$ be a minimal representation-infinite ray-category containing a crown and a profound morphism that is not irreducible. Then there is a profound morphism $\mu$ not occurring in an
 essential contour.
\end{proposition}

 The proof of the proposition takes 15 pages and it is at the moment  the most complicated one mentioned in this survey so that it should be replaced by a better argument. However,
 the proof is similar to
 the proof for
 the existence of a multiplicative basis: the main problem is to find a finite strategy and this consists of a rather boring local part and a rather nice global part.
 All this is  explained well - I hope - in \cite[section 2.2]{Gaps}.

The theorem is then an easy consequence of the proposition. Namely, it follows directly from the construction of the universal coverings that the fundamental groups of $P$ and $P/\mu$
coincide as well as the quivers of $\tilde{P}$ and $\tilde{P/\mu}$ and also $H^{2}(P,Z)$ embeds into $H^{2}(P/\mu,Z)$.

The following fact  can be proven with the above proposition and the finiteness criterion.
\begin{proposition}\label{elementary}
 Let $P$ be a minimal representation-infinite ray-category. Then there is always a profound morphism $\mu$ not occuring in an essential contour.
\end{proposition}

Back in 1983 I tried to prove this directly thereby obtaining a proof for the existence of interval-finite universal coverings based on the representation-finite case. 
I am still  wondering whether there is such a direct proof.

\section{Covering theory}
\subsection{Coverings of translation quivers and k-linear covering functors}
As far as I know the place  where coverings are used for the first time in the representation theory of algebras is the paper 'Group representations without groups' \cite{Groups} by Gabriel and Riedtmann.
They consider coverings of the ordinary quiver, but soon after Riedtmann started her work on representation-finite selfinjective algebras in \cite{Riedtmann} by looking at coverings of the ( stable )
 Auslander-Reiten quiver. This point of view was further developped by Gabriel in \cite{Coverings} where all the following material comes from. My contribution to that article -  stated clearly in the introduction of 
\cite{Coverings} - was only to 
improve some results and to work together on section 6 
on simply connected algebras.

A translation quiver $(\Gamma,\tau)$ is a pair consisting of a locally-finite quiver $\Gamma$ without loops and double-arrows and a bijection $\tau:X \rightarrow Y$ between two subsets of 
$\Gamma_{0}$ such that for all $x$ in $X$ there is an arrow $\alpha:y \rightarrow x$ in $\Gamma $ iff there is an arrow $\sigma \alpha:\tau x \rightarrow y$. We denote the set of these $y$
 as
$(\tau x)^{+}=x^{-}$ and call the vertices in $\Gamma_{0}\setminus X$ projective, in $\Gamma_{0} \setminus Y$ injective. The full subquiver supported by $x,x^{-},\tau x$ is called a mesh and 
the mesh-category 
$k(\Gamma)$ is  the quotient of the path category $k\Gamma$ by the ideal generated by all mesh relations $\sum _{\alpha:y \rightarrow x} \alpha \sigma (\alpha)$ for $x$ in $X$.
The translation quiver is stable if $X$ and $Y$ coincide with $\Gamma_{0}$. The most important examples of ( stable ) translation quivers are ( stable ) Auslander-Reiten quivers.
 For any oriented tree $T$ one has a
 stable translation quiver ${\bf Z} T$. The underlying set is ${\bf Z} \times T$ and the translation is given by $\tau (z,x)=(z-1,x)$. There is an arrow $(z,x) \rightarrow (z',y)$ iff either 
$z=z'$ and there is an arrow $x \rightarrow y$ in $T$ or $z=z'-1$ and there is an arrow $y \rightarrow x$ in $T$. Of course ${\bf Z}T$ does not depend on the orientation of $T$ but only on 
the underlying graph.

A covering of translation quivers is a map $f:\Gamma' \rightarrow \Gamma$  between the quivers with $f \tau' =\tau f$ and such that $x'$ is projective resp.
 injective
iff $fx'$ is so. It is clear how to define a universal covering $\gamma: \tilde{\Gamma} \rightarrow \Gamma$, the fundamental group $\Pi_{\Gamma}$ and simply connected translation quivers. There is the following
 construction of the universal covering. Given a connected translation quiver $(\Gamma,\tau)$ one defines a new quiver $\hat{\Gamma}$ by adding a new arrow $\gamma_{x}:\tau x \rightarrow x$ for each $x$ in $X$. A walk in
 $\hat{\Gamma}$ is a formal composition of old and new arrows and their formal inverses such that the occurring starts and ends fit together well and the composition of walks is defined in the
 obvious way. The homotopy is the smallest equivalence relation stable under left or right multiplication with the same walk and under 'inversion' and such
 that $\alpha\alpha^{-1}$ is equivalent to an identity for each
arrow $\alpha$ - old or new - or each formal inverse and such that $\gamma_{x}$ is equivalent to $\alpha \sigma(\alpha)$ for each arrow $\alpha$ ending in a non-projective vertex $x$.
The fundamental group  is isomorphic to the set of all homotopy classes of walks starting and ending in $x$ endowed with the multiplication induced by the composition of walks.
The universal covering $\tilde{\Gamma}$ has the homotopy classes with start in $x$ as its points. The arrows, the translation and the covering $\pi:\tilde{\Gamma} \rightarrow \Gamma$ are all 
defined in a natural way. 

In contrast to the case of base categories the universal covering has now always good properties which is due to the fact that the homotopy relation is homogeneous provided one gives the new
 arrows $\gamma_{x}$ the degree 2.  There is  a morphism $\kappa$ of translation-quivers from $\tilde{\Gamma}$ to ${\bf Z}A_{2}$   which makes it possible to argue 
 by induction. It follows in particular that all paths between two fixed points have the same length and so $\tilde{\Gamma}$ has 
no oriented cycles.

A k-linear functor $F:M \rightarrow N$ between two locally bounded $k$-linear categories is called a covering functor if it induces for all $m$ in $M$ and $n$ in $N$ isomorphisms
 $$\bigoplus_{Fm'=n} M(m,m') \simeq N(Fm,n)  \, , \,\bigoplus_{Fm'=n} M(m',m) \simeq N(n,Fm).$$ For instance any covering $f:\Gamma' \rightarrow \Gamma$ of translation quivers 
induces a covering functor $k(f):k(\Gamma') \rightarrow k(\Gamma)$ provided both mesh categories are locally bounded.

A locally bounded  category $C$ is an Auslander-category if it is isomorphic to the full subcategory $ind\, A$ of the indecomposables over some locally representation-finite category $A$ or 
- equivalently - if $C$ has global dimension at most 2 and any projective $p$ in $C$ admits an exact sequence $0 \rightarrow p \rightarrow i_{0} \rightarrow i_{1}$ where the $i_{k}$ are 
projective and injective \cite[section VI.5]{Auslander}. One has the following results.
\begin{theorem}\label{well}( \cite{Coverings} )
 Let $A$ be a locally representation-finite category with Auslander-Reiten quiver $\Gamma_{A}$. Then there is a covering functor $$F:k(\tilde{\Gamma}_{A}) \rightarrow ind \,A.$$
\end{theorem}

This is a version of Riedtmanns theorem from \cite[Satz 2.3]{Riedtmann}.
\begin{theorem}\label{coveringfunctor} \cite[proposition 3.5]{Coverings}
 Let $F: C \rightarrow D$ be a covering functor. Then $C$ is an Auslander category iff $D$ is.
\end{theorem}

Combining these two results we see that for any locally representation-finite category $A$ the full subcategory $A^{os}$ of $k(\Gamma_{A})$ consisting of the projective points is locally 
representation-finite. We call this the old standard form of $A$.

\subsection{Galois coverings}
 Around 1980 it was clear to many people that it would be good to have a covering theory that is induced by some group action  and independent of the Auslander-Reiten quiver. There are several 
more or less equivalent ways to obtain such a theory ( see e.g. \cite{Gordon,Green,Green2,delapena2,Waschbsch} ), but I follow Gabriel who presents in \cite[section 3]{Galoiscovering} on 8 pages
 more than we need. This theory was generalized by Dowbor and Skowro\'{n}ski in \cite{Dowbor}.

Let $G$ be a group of $k$-linear automorphisms of a locally bounded category $A$. We assume that $G$ acts freely on $A$ ,i.e. $ga\neq a$ for $g\neq 1 $, and locally bounded,  i.e. for given 
$a,b$ in $A$ there are only finitely many $g$ with $A(a,gb)\neq 0$. Then there is a quotient $F:A \rightarrow A/G$ and this is a covering functor. In fact, a covering functor
$ E:A \rightarrow B$ with $Eg=g$ for all $g$ in $G$ induces an isomorphism $A/G \simeq B$ iff E is surjective on the objects and $G$ acts transitively on all fibres. Such an $E$ is called a
 Galois covering.

Of course, the action of $G$ on $A$ induces an action $m \mapsto m^{g}$ on $A$-mod by scalar extension. This action is free if $m\not\simeq m^{g}$ holds for all $g\neq 1$ and 
$m\neq 0$.

\begin{theorem}\label{Galois}( \cite{Galoiscovering} ) Let $A$ be a locally bounded connected category and $G$ a group of automorphisms acting locally bounded on $A$ and free on $A$ and $A-mod$. Denote by $F$ the quotient and by
$F_{\bullet}:A-mod \rightarrow A/G-mod$ the left adjoint to the restriction.
\begin{enumerate}
 \item $F_{\bullet}$ is exact. It preserves dimensions and indecomposibility.
\item $A$ is locally representation-finite iff $A/G$ is so. In that case $F_{\bullet}$  induces a Galois covering 
$ind A \rightarrow ind (A/G)$ and a covering $\Gamma_{A} \rightarrow \Gamma_{A/G}$ between the Auslander-Reiten quivers.

\end{enumerate}
 
\end{theorem}
Here $F_{\bullet}$ is simply defined  by 'taking direct sums of vector spaces and linear maps' so that the theorem is very helpful. But the problem is to find for a given $B$ a category $A$ 
and a group $G$ with $B\simeq A/G$ such that first the theorem applies and second $A-mod$ has better properties than $B-mod$.

Now the preceding theorem can be applied to the universal covering $\pi:\tilde{P} \rightarrow P$ of a  ray-category if the fundamental group $G$ is free.
Namely  $G$ acts freely and locally bounded on $k\tilde{P}$ and freely on $\tilde{P}-mod$ because $G$ is torsion-free. Thus we see  that $P$ is locally representation-finite iff $\tilde{P}$
 is locally representation-finite. This is very useful once we know that the Auslander-Reiten quiver of $k\tilde{P}$ or large portions thereof can be easily determined.
This will be guaranteed in many cases by the next results.

\subsection{Coverings of ray-categories and of Auslander-Reiten quivers}
Let $A$ be a directed Schurian algebra. For any $a$ in the quiver $Q$ of $A$ let $Qa$ be the full subquiver consisting of the points $b$ such that there is no path from $b$ to $a$. 
The point $a$ is separating if the supports of different indecomposable direct summands of the radical of $A(a, \,)$ lie in different connected components of $Qa$. If $\tilde{P}$ is an
 interval-finite universal covering of a ray-category, then each finite convex subcategory $A$ of $k\tilde{P}$ satisfies these assumptions. As a slight generalization
 of 
the separation criterion due to Bautista-Larri\'{o}n one has:

\begin{theorem}\label{Trennung}( \cite{Larrion,Criterion} )
 Let $A$ be a ( connected ) Schurian directed algebra such that each point is separating.
\begin{enumerate}
 \item If $A$ is representation-finite it has a finite simply connected Auslander-Reiten quiver.
\item If $A$ is minimal representation-infinite $\Gamma_{A}$ has a simply connected component consisting of the $\tau^{-1}$-orbits of all indecomposable projectives.
\end{enumerate}
\end{theorem}

Next we prove amongst other things that the two definitions of the standard form coincide for a locally representation-finite category. This was shown for the first time in  \cite{Bretscher}.

\begin{theorem}\label{GQARQ}
 Let $A$ be a locally representation-finite category with associated ray-category $P$ having universal covering $\pi: \tilde{P} \rightarrow P$. Then we have:
 \begin{enumerate}
\item $\Gamma_{k\tilde{P}}$ is simply connected.
  \item $\pi$ induces a universal covering $\pi':\Gamma_{k\tilde{P}} \rightarrow \Gamma_{kP}$. Furthermore, $\Pi_{P}$ is isomorphic to $\Pi_{\Gamma_{kP}}$.
\item $kP$ is isomorphic to the full subcategory of $k(\Gamma_{kP})$ formed by the projective points.
\item $A^s$ and $A^{os}$ are isomorphic.
 \end{enumerate}

\end{theorem}

\begin{proof}a) Recall that $P$ and therefore also $\tilde{P}$ are supposed to be connected. We know that $A$ and its standard category $A^{s}=kP$ have the same Auslander-Reiten quiver 
from theorem \ref{faithful}. Furthermore $\tilde{P}$ is interval-finite and the fundamental group $\Pi_{P}$ is free by  theorem \cite{Fischbacher}. Finally, $\tilde{P}$ 
and $P$ are  locally representation-finite by theorem \ref{Galois} and each finite connected convex subcategory $C$ of $\tilde{P}$ has a simply  connected Auslander-Reiten quiver
 by the last theorem.

Now one shows with the usual arguments that the Auslander-Reiten quiver of a connected locally representation-finite category is connected. Thus $\Gamma_{k\tilde{P}}$ is connected. To show 
that it is even simply connected let $x,y$ be two points in $\Gamma_{k\tilde{P}}$ and $v,w$ two paths from $x$ to $y$ consisting of irreducible maps. Then there are only finitely many modules
 occuring
 in $v$ and $w$ and there is a finite convex connected subcategory $C$ of $\tilde{P}$ such that all these modules are in fact $C$-modules and the two paths are two paths formed by 
irreducible maps in
$C$-mod. But $C$-mod has a simply connected Auslander-Reiten quiver  and so the two paths have the same length. This implies that $\Gamma_{k\tilde{P}}$ is interval-finite.
 Now, let $w$ be any closed walk in $\hat{\Gamma}_{k\tilde{P}}$. It passes through only finitely many points of $\Gamma_{k\tilde{P}}$ which lie in a finite convex subset that consists
 only of $C$-modules 
for another finite connected convex subcategory of $\tilde{P}$. The original walk is a walk in $\hat{\Gamma}_{C}$ and so it is null-homotopic. Here one uses the  fact that in a finite 
simply connected translation quiver any walk $w$ can be contracted to a point within the quiver $\hat{\Delta}$ corresponding to the convex hull $\Delta$ of the points occuring in $w$.

b) Of course $\pi$ induces a Galois covering $k\pi:k\tilde{P} \rightarrow kP$ that in turn induces a covering $\pi':\Gamma_{k\tilde{P}} \rightarrow \Gamma_{kP}$ by theorem \ref{Galois}.
 This is isomorphic to the universal
 covering 
$\gamma:\tilde{\Gamma}_{kP} \rightarrow \Gamma_{kP}$ because $\Gamma_{k\tilde{P}}$ is 
simply connected. Now any $g$ in $\Pi_{P}$ gives rise to an automorphism of $\pi'$. The resulting homomorphism $\Pi_{P} \rightarrow \Pi_{\Gamma_{kP}}$ is bijective, because $\Pi_{P}$ acts
 simply transitive 
on the fibres of $\pi'$. 

c) We have $k(\Gamma_{k\tilde{P}}) \simeq ind \,k\tilde{P}$ by theorem \ref{well}. The quotient $k(\Gamma_{k\tilde{P}}) \rightarrow k(\Gamma_{kP})$ induces a quotient between the full
 subcategories formed by the projective points and this is the wanted isomorphism.

d) This  follows from c) because $A$ and $A^s$ have the same Auslander-Reiten quiver.
\end{proof}

\section{Two classification results}

\subsection{The structure of large indecomposables over simply connected representation-finite algebras}

In  this section we study the indecomposables over  a representation-finite algebra with simply connected Auslander Reiten quiver. The support $A$ of any indecomposable $U$ is then a convex subcategory by 
\cite{quadratic}, whence again simply connected, and $U$ is a sincere  $A$-module. So it suffices to classify the representation-finite algebras with simply connected 
Auslander-Reiten quiver having a sincere indecomposable together with all these modules. Call these algebras ssc in the sequel and denote by $n(A)$ the number of points of an algebra.

In \cite{Treue} I published a list of 24 infinite families of algebras together with some modules that contains all ssc algebras and their sincere indecomposables up to duality and up to
 some
 exceptional algebras with $n(A)\leq 72$. At the end of 1982 I had
 determined by computer also the exceptional algebras and their sincere indecomposables  and verified that they occur only for $n(A)\leq 13$.  I never published these results 
 because the exceptional algebras are not very important and  my results consisted in unreadeable computer-lists only. It is however very remarkable that there are only finitely many ways 
to construct an indecomposable over an ssc algebra and the only way I know how to prove this is to do the classification of the large ssc algebras. 

The complete list of these algebras as given e.g. in \cite[section 10.6]{Buch} does not really matter. But the next result that one can verify by a look at the list plays 
an important  role later on in the proofs of the finiteness criterion and of BT 2. In a certain sense the truth of the BT 2 conjecture is equivalent to the first property of ssc algebras
mentioned in the next result.

\begin{theorem}\label{ssc}
 Let $A$ be an ssc algebra  having an indecomposable of dimension $n\geq 1000$. Then $A$ contains a line with at least $\frac{n}{6}$ points. Moreover, $A$ is the union of at most three lines.
\end{theorem}

  The simple method to obtain the 24 families is  the inductive procedure to construct 
ssc algebras via one-point extensions as described in \cite[section 6]{Coverings}. The results \cite{Nazarova,Kleiner,Kleiner2} of Nazarova, Roiter and Kleiner on 
representations of partially ordered sets play an essential role and the proof is a 
nice interplay between Auslander-Reiten theory and results of the Kiev-school similar as in \cite{Ringel}.

The 24 families as well as the bound $13$ were later verified by Ringel in the last chapter of the book \cite{Tame} by a different method and without computer. This method  was also used by 
 Dräxler to confirm my
 numerical 
results about the exceptional algebras in \cite{Drxler}. Finally two of my students produced pictures showing the zoo of exceptional algebras in \cite{RogatTesche}.

Dräxler used his numerical results to derive the following interesting fact \cite{Nor}:
\begin{theorem}\label{Nulleins}
 Any indecomposable over a ssc algebra admits a basis such that all arrows are represented by matrices with only $0$ and $1$ as entries.
\end{theorem}

We also need the next result from \cite{Extensions} which is proven by elementary algebraic geometry. 
\begin{theorem}\label{extfin}
 Any indecomposable over a ssc algebra is accessible.
\end{theorem}

\subsection{The critical algebras and tame concealed algebras}
A Schurian directed algebra $A=k\vec{A}$ with $H_{1}\vec{A}=0$ is called critical  if $A$ is not representation-finite, but any proper convex subalgebra  is. By theorem
 \ref{Trennung},
any critical algebra has a simply connected component $C$ in its Auslander-Reiten quiver containing all indecomposable projectives. It follows easily that $C$ is a full subquiver of ${\bf Z}T$ for some
 tree $T$. 
\begin{theorem}\label{Euclidean}
 The only trees occurring for critical algebras are of type $\tilde{D}_{n}$, $n\geq 4$, or $\tilde{E}_{n}$, $6 \leq n \leq 8$.
\end{theorem}

I announced this theorem at Luminy in 1982 and the proof as well as the possible algebras for $\tilde{D_{n}}$ appeared in \cite{liste}. The quite technical method for the proof is the 
same as in section 6.1. In 1983 I also determined by computer the critical algebras for the types $\tilde{E}_{n}$, $n=6,7,8$
 using my  previously obtained results for the ssc algebras.
In particular I never used tilting theory as claimed in \cite[page 247]{Skowronskisimson}

In parallel work \cite{Listebild} Happel and Vossieck studied the representation-infinite algebras $B$ having a preprojective component and such that the quotient $B/BeB$ is representation-finite for each 
non-zero idempotent $e$. Using a theorem of Ovsienko \cite{Ovsienko} they derived the analogue of theorem \ref{Euclidean},  namely 
that  with the exception of  a generalized Kronecker-algebra $B$ is the endomorphism algebra of a preprojective tilting
 module over 
an extended Dynkin-quiver. Furthermore they classified 
all these so-called tame concealed algebras with the help of a computer by certain frames which give a complete description by quiver and relations. For types different from  $\tilde{A}_{n}$ 
the algebras studied by 
Happel and Vossieck are obviously critical and the two classes even coincide just because the numbers of isomorphism classes coincide. In fact, this is also clear for theoretical reasons 
if one knows that for critical algebras
 only extended Dynkin-quivers occur. As far as I know  there is no simple proof for that even though von Höhne has verified this by hand in \cite{vonHhne}. We state for later use:

\begin{theorem}\label{crit}
 The critical algebras are given by the frames of Happel and Vossieck.
\end{theorem}

We do not reproduce the frames that can be found in the original articles and  many other places e.g. in \cite{Buch}. We will refer to the list as the BHV-list.

Besides theorem \ref{Nulleins} this list is the only result in this survey that depends on computer calculations. We leave it up to the reader  whether this makes it trustworthy or not. The only place where
 the correctness of this list matters is
in the finiteness criterion given later on. In the other 'applications' e.g. in the proofs of section 3 one uses only the fact that some very special members of the list like 
certain tree-algebras or fully commutative quivers 
are not representation-finite. This can always be checked easily and it was known since a long time.

Later on we  also use the following non-trivial fact from \cite{Degenerations} which is again proven with geometric methods like degenerations of modules.

\begin{theorem}\label{tameext}
 Over a tame concealed algebra any indecomposable is accessible.
\end{theorem}.

\section{Some applications}
\subsection{On indecomposables over representation-finite algebras}
 By theorem \ref{faithful} any representation-finite  algebra with a  faithful indecomposable $U$ is standard. Let $P=\vec{A}$ be the ray-category and $\pi:\tilde{P} \rightarrow P$ 
the universal covering
 inducing the Galois covering
$F:k\tilde{P} \rightarrow kP\simeq A$.
 By theorem \ref{Galois} there is a $k\tilde{P}$-module $\tilde{U}$ with $F_{\bullet}\tilde{U}\simeq U$. 
Thus there are up to Galois covering only finitely many ways to construct an indecomposable over a representation-finite algebra. In particular we obtain from theorems \ref{extfin} and \ref{Nulleins}:
\begin{theorem}\label{Strucind}
 Any indecomposable over a representation-finite algebra is accessible and it admits a basis such that all arrows are represented by matrices containing only $0$ and $1$ as entries.
\end{theorem}
 For further applications the following facts are useful. Here an infinite line $L$ in the universal covering $\tilde{P}$ of a ray-category is called periodic if the stabilizer $G_{L}$ 
 within the fundamental group is not trivial. The periodic length of $L$ is then the number of $G_{L}$-orbits on $L$ ( see \cite{Dowbor} ).
\begin{lemma}\label{line}
 Let $\pi:\tilde{P} \rightarrow P$ be the universal covering of a  ray-category with $d$ elements such that  $\tilde{P}$ is interval-finite and weakly zigzag-finite.

\begin{enumerate}
 \item Let  $ q \rightarrow z_{1} \ldots  \rightarrow z_{t-1} \rightarrow z_{t} \leftarrow \ldots  x' \leftarrow q'$ be a line $L$ of length $l$ in $\tilde{P}$ with $\pi(z_{1})\neq \pi(x')$ and 
$\pi(q)=\pi(q')$ and such that $v=q \rightarrow z_{1} \ldots  \rightarrow z_{t-1} \rightarrow z_{t}$ is a path. Let $g$ be the element in the fundamental group  with $gq=q'$. 
Then  adding on the 
right end the  path $gv$ one gets another convex subcategory in $\tilde{P}$. Continuing that way in both directions one obtains an infinite periodic line $L'$ of periodic length $l-1$ at most.
\item Any line of length $2d$ contains a 'subline' $L$ that can be prolonged to an infinite periodic line of length at most $2d$.
\item If $P$ is minimal representation-infinite and $\tilde{P}$ contains an infinite periodic line then $P$ is a zero-relation algebra.
\end{enumerate}
\end{lemma}

\begin{proof}  a) That one gets a convex subcategory by adding $gv$ is an  easy consequence of the fact that all points of  $\tilde{P}$ are separating by theorem \ref{separation}. Then one dualizes and adds $gw$ on the right
 end where $w\neq v$ is the second subpath of $L$ ending in the sink $z_{t}$. This procedure goes on forever to the right side and after that one goes to the left. By construction
$g$ is in the stabilizer 
of the obtained line $L'$ and the periodic length is $d-1$ at most.

 b) Take a line $L$ of length $2d$ and mark all sources $q_{1},q_{2},\ldots  ,q_{r}$. Thus $L$ looks like
 $$    \ldots   \leftarrow q_{1} \rightarrow \ldots \leftarrow q_{i}\rightarrow  \ldots \leftarrow q_{r}\rightarrow \ldots$$ and we get $2d \leq \sum _{j=1}^{r} dim P(q_{j}, \, )$. So there
 are three sources  mapped onto the same point in $P$. After renumbering we can assume they are $q_{1},q_{i},q_{r}$. For two of them the situation of part a) occurs and we are done.

c) Let $L$ be a periodic line in $\tilde{P}$ containing the source $q$. For each natural number $n\geq 1$ let $L_{n}$ be the subline of $L$ of length $n$ starting in $q$ and going to the right side.
 Then there is an indecomposable 
representation $U_{n}$ of $\tilde{P}$ with support $L_{n}$ and dimension $n$. The pushdowns $F_{\bullet}U_{n}$  are all indecomposable and annihilated by all contours. Infinitely many of 
them are faithful because $P$ is minimal representation-infinite. Thus $P$ is defined by zero-relations.
\end{proof}

Using parts of the lemma, the detailed structure of the large ssc algebras and $k$-linear covering functors I proved in \cite{Treue}:

\begin{theorem}\label{bounds}
 Let $A$ be a basic representation-finite algebra of dimension d. Let $u(A)$ be the number of indecomposable $A$-modules and let $u(d)$ be the supremum of the $u(A)$ when $A$ runs
 through all representation-finite algebras of dimension $d$. Then we have:
\begin{enumerate}
 \item An indecomposable $A$-module has at most dimension $max \{2d,1000\}$.
\item There is a constant $C$ such that for all $d\geq 4$ one has $2^{\sqrt{d}} \leq u(d) \leq 9d^{6}\cdot 2^{2d+7} +Cd.$
\end{enumerate}

\end{theorem}
The mere existence of  upper bounds was proved in \cite{JensenLenzing} by methods from model theory. To make these bounds concrete was back in 1981 one of my motivations for the
 classification of the ssc algebras.

\subsection{A criterion for finite representation type} 
Given a finite dimensional algebra $A$ as a subalgebra of some endomorphism algebra or by generators and relations it might be impossible to find a quiver $Q$ and an admissible ideal $I$ 
such that $A$ and $kQ/I$ are Morita equivalent. So we assume in our criterion right from the beginning that $A=kQ/I$ holds and that we know the dimension $d$ of $A$. 
Then it is easy to find out whether $A$ is distributive and to determine the ray-category $\vec{A}$ and one has the following  criterion for finite representation type.\newpage

\begin{theorem}\label{criterion}
 Let $A= kQ/I$ be a connected  distributive algebra of dimension $d$ given by a quiver and an admissible ideal. Let $\pi:\tilde{P} \rightarrow P$ be the universal covering of $P=\vec{A}$. 
Then $A$ is representation-finite iff it satisfies the following two conditions:
\begin{enumerate}
\item $P$ contains no zigzag of length $2d$.
\item $\tilde{P}$ contains no algebra of the BHV-list as a full convex subcategory. 
\end{enumerate}
 
\end{theorem}

\begin{proof} By theorem \ref{faithful} we know that $A$ is representation-finite iff $P$ is so.

For the easy implication let $P$ be representation-finite. A zigzag $Z$ of length $2d$ in $P$ would contain  one  non-invertible morphism at least three times. Thus we find a crown in $P$ which is impossible.
So $P$ is zigzag-finite. By theorems \ref{Fischbacher} and  \ref{Galois} one gets that $\tilde{P}$ is locally representation-finite and so it contains no 
algebra of the BHV-list as a convex subcategory. Here we do not need any classification but only the fact that all algebras in the BHV-list are representation-infinite.

Reversely, $P$ is zigzag-finite whence $P$ is representation-finite iff $\tilde{P}$ is locally representation-finite by theorems \ref{Fischbacher} and \ref{Galois} again. Observe that 
$\tilde{P}$ is interval-finite and satisfies $H_{1}\tilde{P}=0$. 

If $\tilde{P}$ is not locally 
representation-finite there are two cases possible. First assume that there is a finite convex subcategory $B$ which is not representation-finite. Then one finds also a critical convex 
subcategory $C$ by removing certain sinks or sources of $B$.  Then $C$ is an algebra of the BHV-list by theorem \ref{crit} which is a contradiction.

 If all finite convex subcategories are representation-finite then
 there is a point $x$ in $\tilde{P}$ such that there are infinitely many indecomposables $U$ with $U(x)\neq 0$. The supports of these modules can get arbitrarily large. By theorem \ref{ssc} the 
 convex support of an
 indecomposable of dimension at least $12d + 1000$ contains a line of length $2d$. Thus $\tilde{P}$ contains an infinite line by lemma \ref{line}, whence a zigzag $Z$. Its push-down
 $F_{\bullet} Z$ is a zigzag in $P$ which is 
again a contradiction. Note that 
for this implication we need both classification results.
\end{proof}

I anounced a criterion similar to the one above in 1982 at Luminy. Of course I had to make more assumptions because a lot of theorems entering the proof above were not yet proven at that time.
However also the original criterion \cite{Criterion} needed more than the first Jans condition and preprojective tilting.

 Fischbacher used the criterion  in \cite{Fisch3}
 for the classification of all maximal representation-finite and minimal representation-infinite algebras with three points. If the reader has the energy to apply the criterion to some cases
where the fundamental group is free in two generators and the algebra is not defined by zero-relations he will appreciate very soon that one has to look  at convex subcategories 
of $\tilde{P}$ only. A detailed example is given in \cite{Criterion}.

The criterion leads to  an algorithm  \cite{delapena1} that decides in 'polynomial' dependence of the dimension of $A$ whether $A$ is representation-finite or not even though the
 number of indecomposables can grow exponentially by theorem \ref{bounds}.

We end this section with the following statement whose proof is left to the reader. The field $k$ is here arbitrary.
\begin{theorem}\label{promild} Let $P$ be a ray-category.
\begin{enumerate}
 \item If $P$ is minimal representation-infinite it is finite.
\item If $kP$ is locally representation-finite for one field it is so for all fields. In that case the Auslander-Reiten quivers, 
the dimension-vectors of the indecomposables etc. are independent of $k$.
\end{enumerate}
\end{theorem}
\subsection{The proofs of Brauer Thrall 0 and Brauer Thrall 2}
Now we prove the sharper version of BT 0 proposed by Ringel. 

\begin{theorem}\label{BT 0}
 Let $A$ be a basic finite-dimensional algebra. If there is an indecomposable of dimension $n$ there is also an accessible module of that dimension.
\end{theorem}
\begin{proof} For representation-finite algebras all indecomposables are accessible by theorem \ref{Strucind} and for non-distributive algebras the theorem holds by Ringels result from section 2.2.
 Thus we can assume
 that $A$ is distributive and minimal representation-infinite whence standard by theorem \ref{faithful}. The universal covering $\tilde{P}$ is not locally representation-finite and interval-finite 
with free fundamental group by theorem \ref{Fischbong}.
If  we find a critical algebra $B$ as a convex subcategory any indecomposable $B$-module is accessible by theorem \ref{tameext}. Their images under $F_{\bullet}$ provide 
accessible modules in all dimensions. If all finite subcategories of $\tilde{P}$ are representation-finite  there are arbitrarily large indecomposables over a representation-finite
 convex subcategory. These indecomposables are accessible by theorem \ref{Strucind} and again their images under $F_{\bullet}$
give accessible modules in all dimensions.
\end{proof}

Observe that this proof uses none of the two classification results. Moreover theorem \ref{gaps} ( but not theorem \ref{BT 0} ) remains valid if $k$ is a splitting field for $A$, i.e. if all simple $A$-modules of
 finite dimension have trivial endomorphism algebra $k$.

Our version of Brauer-Thrall 2 reads as follows:

\begin{theorem}\label{BT2}
 Let $A$ be a basic representation-infinite algebra of dimension $d$. Then there is a natural number $e\leq \max\{30,4d\}$ and pairwise non-isomorphic $e$-dimensional 
indecomposables $U_{i}$, $i\in k^{\ast}$, 
 such that for any $n\geq 1$ there exist pairwise non-isomorphic indecomposables $U_{n,i}$ having a chain of $n+1$ submodules such that all successive quotients are isomorphic to $U_{i}$.
\end{theorem}
\begin{proof}
 
 We can assume that $A$ is minimal representation-infinite. 

If $A$ is not distributive there are two idempotents $e,f$ and two linearly independent elements  $v,w$ in $eAf$ annihilated by the radical of $A$.
For any natural number $n$ and any element $i$ in $k^{\ast}$ we consider the morphism $$\phi(n,i): (Ae)^{n} \rightarrow (Af)^{n}$$ given by the matrix $vE_{n} + w( iE_{n}+ N_{n})$ and its cokernel
$U(n,i)$ of dimension $n( dim Af -1)$. Here $E_{n}$ is the $n\times n$-identity matrix and $N_{n}$ a nilpotent $n\times n$ Jordan-bloc. A simple direct calculation  shows that the $U(n,i)$ are pairwise non-isomorphic indecomposables admitting an exact sequence 
 $0 \rightarrow U(n-1,i) \rightarrow U(n,i) \rightarrow U(1,i) \rightarrow 0$.

If $A$ is distributive and minimal representation-infinite it is standard by theorem \ref{faithful} and we call its ray-category $P$. By theorems \ref{Fischbong} and \ref{Galois} the universal covering $\tilde{P}$ is interval-finite with 
free fundamental group and $\tilde{P}$ is not locally representation-finite. Thus $\tilde{P}$ contains a critical convex subcategory $C$ or all finite convex subcategories are
 representation-finite whence $\tilde{P}$ contains a line of length $2d$ by theorem \ref{ssc}.

First let $C$ be a convex subcategory of $\tilde{P}$ which is tame concealed of type $\tilde{E}_{n}$.  The wanted families of modules are then the push-downs of the obvious
 ${\bf P}^{1}(k)$-families of indecomposable $C$-modules 
 with dimension-vector $n\delta$ where $\delta$ is the null-root of $C$.
 The dimension of each $U_{i}$ is the sum of the components of the null-root whence 
smaller than $30$.

Next assume that there is  a periodic line $L$ in $\tilde{P}$ of periodic length $e\leq 2d$. Then $A$ is a zero-relation algebra by lemma \ref{line}.
To produce the wanted modules one can invoke the theory of Dowbor and Skowro\'{n}ski \cite{Dowbor} or one can observe with Ringel \cite{Rinmin} that $A$ is special
 biserial.  One always finds the wanted modules as appropriate band-modules with $dim U_{i}\leq 2d$. By part b) of lemma \ref{line} we can assume from now on that $\tilde{P}$ does not contain 
a line of length $2d$.

The only case not yet settled is when $\tilde{P}$ contains a tame concealed algebra $C$ of type $\tilde{D}_{n}$ as a convex subcategory. Any such subcategory contains a line of length $n-1$
 so that we can assume  $n-1 \leq 2d$.  Then the push-downs of the obvious  ${\bf P}^{1}(k)$-families of $C$-modules produce indecomposables with $dim \,U_{i} \leq 4d$.

\end{proof}

The proofs given in \cite{Bautista,Todorov,fIbt} of the usual weaker form of BT 2 do not need theorem \ref{Fischbong} or the classification of the critical algebras, but the classification of the large ssc algebras
 is always needed. With a little more work one finds in all three cases occurring in the last proof a natural ${\bf P}^{1}(k)$-family of modules. 

\subsection{Finite representation type is open} 

In the important early paper \cite{Open} Gabriel introduces for fixed natural numbers $d$ and $n$ the varieties $alg_{d}$ of $d$-dimensional unital associative algebras and $alg_{d}mod_{n}$ of pairs
 consisting 
of a $d$-dimensional algebra $A$ and an $n$-dimensional $A$-module $M$. These varieties are equipped with natural actions of $Gl_{d}$ and $Gl_{n}$ and the projection
$\pi:alg_{d}mod_{n} \rightarrow alg_{d}$  has the following property.
\begin{proposition}
 The image of a closed $Gl_{n}$-stable subset of $alg_{d}mod_{n}$ is closed.
\end{proposition}

From this Gabriel derives with the help of some semi-continuity properties of fibres:
\begin{proposition}
 For any $n$ the set $U(n)$ of all $d$-dimensional algebras having only finitely many isomorphism classes of $n$-dimensional modules is open.
\end{proposition}
Finally one gets:
\begin{theorem}\label{Finopen}
 The set $fin_{d}$ of all representation-finite algebras of dimension $d$ is open in $alg_{d}$.
\end{theorem}

For the proof Gabriel observes that $fin_{d}$ is the intersection of all $U(n)$ by BT 2 and a countable union of constructible sets by Auslanders homological characterization of 
representation-finite algebras. Then one has the surprising fact that the intersection as well as the union stop at a finite level which implies that there are constants $C_{1}$ and $C_{2}$
such that all representation-infinite algebras of dimension $d$ have infinitely many isomorphism classes of indecomposables of dimension $\leq C_{1}$ and all representation-finite algebras
of dimension $d$ have at most $C_{2}$ isomorphism classes of indecomposables. Of course these results are now surpassed by theorem \ref{BT2} and theorem \ref{bounds}.

Theorem \ref{Finopen} is used in \cite{BGRS} for the proof of part a) of theorem \ref{faithful}. 
 
Geiss has combined in \cite{Geiss} Gabriels arguments with Drozd'theorem on tame and wild algebras\cite{Drozd,CBD,Zahmwild} to show.

\begin{theorem}
 Any deformation of a tame algebra is tame.
\end{theorem}

This result is very useful because a lot of interesting algebras with unknown module structure degenerate to special biserial algebras which are always tame \cite{GP,Wald}. 
For instance this is used in the 
recent interesting results of Brüstle, de la Pe\~{n}a and Skowro\'{n}ski on tame strongly simply connected algebras \cite{Brstle}. Unfortunately, one does not get a description of the 
indecomposables with this method.

The theorems of Gabriel and Geiss have been analyzed and generalized in \cite{CB,Kasjan}, but the question whether tame type is open is not yet answered.

\subsection{The 'classification' of representation-finite and minimal representation-infinite algebras}

For representation-finite algebras there is no classification in the strict sense known. 

We have already seen several times that representation-finite algebras behave worse than minimal representation-infinite algebras. So one might ask  whether these algebras
 can be classified. If we look  at  a distributive algebra $A$ then we have by theorems \ref{faithful} and \ref{Fischbong} that $A\simeq kP$ for some ray-category $P$ with an 
interval-finite universal covering $\tilde{P}$ that is not 
locally representation-finite. 

If all finite convex subcategories are representation-finite then $A$ is a zero-relation algebra by lemma \ref{line} and in fact even a  special biserial
 algebra as observed in \cite{Rinmin} where Ringel classifies the minimal representation-infinite ones among these  and where he analyzes their  module categories. 

On the other hand the 
case where $k\tilde{P}$ contains a 
tame concealed algebra of type $\tilde{E}_{n}$ with $n=6,7,8$ leads to algebras with at most 9 points and thus to a finite classification problem. However Fischbachers list 
in \cite{Fisch3} shows that this classification  
 will probably end up with an unreadable list. 

 By the way,  'mild can be wild' i.e. there are a lot of wild algebras that are minimal representation-infinite so that in general one will not  get the structure
 of the modules.

\subsection{The classification of representation-finite selfinjective algebras}

The classification of the blocks with cyclic defect and their indecomposable modules due to Dade, Janusz and Kupisch \cite{Dade,Janusz,Kupisch,Indecom} is an early highlight of the representation theory 
of algebras. Also the detailed study of representation-finite symmetric algebras undertaken by Kupisch in \cite{Kupisch1,Kupisch2} is quite impressing, but his results formulated in terms of Cartan 
numbers are difficult to understand and far from 
a complete classification.

The situation changed after the invention of almost split sequences when Riedtmann discovered in \cite{Riedtmann} that the stable Auslander-Reiten quiver of a representation-finite  
algebra has a
 simple structure ( see also Todorovs work in \cite{Todo} ). For  a selfinjective algebra only the missing projective-injectives have to be inserted and the possible configurations of  these points were first studied purely combinatorially by Riedtmann, but later on Hughes-Waschbüsch \cite{Hughes} as well as
Bretscher-Riedtmann-Läser \cite{BLR} found independently two direct constructions and the second group of authors classified that way the configurations and so also the representation-finite
 selfinjective standard algebras. 

However, Riedtmann had observed earlier that for some configurations and only in characteristic 2 there is also an exceptional non-standard algebra. 
The published proofs that 
this is the only 'accident' are complicated. In the approach by  Riedtmann, Bretscher and Läser it is contained in the articles
 \cite{RiedtmannAn,Riedtmann2,BLR}, in the approach by Kupisch,Waschbüsch and Scherzler in the articles \cite{Kupisch1,Kupisch2,Kupisch3,Waschbsch1}. Some of  these  are long  and 
difficult to read.

Here, to illustrate the strategy of section 3 we will analyze which critical paths and non-deep contours occur for representation-finite selfinjective algebras by using the classification of the 
standard-algebras. It turns out that the difficult steps 1 and 2 of section 3  will never occur except for one case where we deal with one penny-farthing glued to a Brauer-quiver algebra. Thus one can get a complete
proof of the classification of the representation-finite selfinjective algebras by reading \cite{Riedtmann,BLR} and \cite[sections 8.3-8.6]{BGRS}.

Let $T$ be a Dynkin-diagram. For any subset $C$ of ${\bf Z}T$ we denote by ${\bf Z}T_{C}$ the simply connected 
translation quiver
 obtained from ${\bf Z}T$ by adding for each $c \in C$ a new point $c^{\ast}$ which is projective and 
injective and arrows 
$c \rightarrow c^{\ast} \rightarrow \tau^{-1}c$. We call $C$ a configuration of $T$ if ${\bf Z}T_{C}$ is the Auslander-Reiten quiver of a locally representation-finite selfinjective
 category $A=A(T,C)$ which is then uniquely determined because of  $k({\bf Z}T_{C}) \simeq ind \,A(T,C)$ ( see theorem \ref{well} ). The period $e$ of a configuration is the smallest natural number such
 that $\tau^{e}$ stabilizes $C$.

 We set $m(A_{n})=n, m(D_{n})=2n-3,m(E_{6})=11,
m(E_{7})=17,m(E_{8})=29$. The following crucial result follows easily from the well-known properties of the additive functions in $k({\bf Z}T)$ given in \cite{Gabaus}.

\begin{lemma}\label{m(T)}
 Let  $C$ be a configuration of $T$. Let $S$ be a simple $A$-module with projective covering $P(S)=c^{\ast}$ and injective hull $I(S)=d^{\ast}$. Then we have:
 \begin{enumerate}
\item $Hom (P(S),I(S))=k$ and the only profound morphisms in $ind \,A(T,C)$ are the non-zero elements in $Hom (P(S),I(S))$.

  \item $d=\tau^{-m(T)}c$, whence $C$ is stable under $\tau^{m(T)}$ and $e$ divides $m(T)$.
\item $k({\bf Z}T_{C})(c^{\ast},\tau^{r}c^{\ast})= k({\bf Z}T_{C})(c^{\ast},\tau^{s}\phi c^{\ast})=0$ for all $r > m(T)$, $s \geq m(T)$ and non-trivial automorphisms $\phi$ of $T$.
 \end{enumerate}

\end{lemma}

\begin{proof} For any  locally bounded category the  profound morphisms in $ind \,A$ are  of the form $P(S) \rightarrow S \rightarrow I(S)$. The rest follows from the shape of the additive functions. 
 Details can be found in \cite[proposition 1]{BLR}.
\end{proof}

Now let $A$ be a representation-finite selfinjective algebra with Auslander-Reiten quiver $\Gamma_{A}$. By Riedtmanns structure theorem  \cite{Riedtmann} its universal 
covering $\tilde{\Gamma}_{A}$ is isomorphic to some 
${\bf Z}T_{C}$ for a uniquely determined Dynkin-diagram $T$ called the tree class of $A$ and a configuration $C$ of  $T$. The standard-form $A^{s}$ is then the full subcategory of projectives 
in $k({\bf Z}T_{C}/\Pi)$ where $\Pi$ is an admissible subgroup of $Aut \,{\bf Z}T$ leaving $C$ invariant and one obtains that way all standard-algebras up to isomorphism. The group $\Pi$ is the fundamental group of $\vec{A}$ and its structure is determined in 
\cite[3.3]{RiedtmannAn} and \cite[section 1.6]{BLR}.

\begin{lemma}\label{fundamental} The fundamental groups are always cyclic and we are in one of the following two situations:
\begin{enumerate}
 \item $\Pi= \langle \tau^{r}\phi \rangle$ for some $r\geq m(T)$ and some ( possibly trivial ) automorphism $\phi$ of $T$. Then we call $\Pi$ small.
\item $\Pi= \langle \tau^{se} \rangle$ for some $s\geq 1$ such that $se<m(T)$. Then we call $\Pi$ large. This occurs only for the types $A_{n}$ and $D_{n}$.
\end{enumerate}
\end{lemma}
An analogue of  the next lemma is proven for the old definition of 'standard' in \cite{BLR}. We include a proof because we want to know which non-deep contours occur and whether all
morphisms are bitransit. Kupisch calls 
an algebra regular if all morphisms are bitransit.  There is no critical path in a regular algebra.
\begin{lemma}\label{small} If the fundamental group of $\vec{A}$ is small any morphism is bitransit and there is no non-deep contour. Therefore $A$ is standard.
 
\end{lemma}
\begin{proof} Lemma \ref{m(T)} implies that all non-invertible endomorphisms of an indecomposable projective are annihilated by all irreducible morphisms. Thus all morphisms are bitransit and 
there is no non-deep contour. An application of step 3 in the   the proof
of theorem \ref{multbasis} shows that $A$ is standard.
\end{proof}

For the case of large fundamental groups we need more information.
By part b) of lemma \ref{m(T)} the period $e$ of a configuration always divides  $m(T)$. Thus to describe all possible configurations means to classify all standard algebras with fundamental group
$\langle \tau^{m(T)} \rangle$. This is done in the sections 6-8 of the   article \cite{BLR}  whose sections 2-4 can be proven much shorter and without using tilting theory.

For type $A_{n}$ one gets the well-known Brauer-quiver algebras. A Brauer-quiver $Q$ is a finite connected quiver such that: i) Each arrow belongs to a simple oriented cycle,
ii) Each vertex belongs to exactly two cycles and iii) Two cycles meet in one vertex at most. It follows that the cycles can be divided into two types say $\alpha$ and $\beta$ such that 
cycles meeting in a vertex belong to different types. Each arrow belongs only to one cycle so that we can speak of $\alpha$-arrows and $\beta$-arrows. The relations that a 
Brauer-quiver algebra has to satisfy are i) The compositions of two arrows of different types is 0 and ii) $\alpha^{a_{x}}= \beta^{b_{x}}$ holds for any point $x$ belonging to an $\alpha$-cycle
of length $a_{x}$ and a $\beta$-cycle of length $b_{x}$. These relations do not define an admissible ideal since there are always loops. For  examples we refer to \cite[section 6]{BLR}. 

The  algebras of type $D_{n}$ are obtained by a kind of glueing of two or three Brauer-quiver algebras. In the first case 
 the fundamental groups are small. Thus we only 
sketch the construction of \cite[section 7.2]{BLR} for the glueing of three Brauer-quiver algebras and for $n\geq 5$.

So let 
$5\leq n= n_{1}+ n_{2}+ n_{3}$ be given and let $P_{i}$, $i=1,2,3$  be Brauer-quiver algebras with $n_{i}$ vertices including a distinguished vertex $x_{i}$ lying on a $\beta$-loop.
To get the quiver of the algebra $D(P_{1},P_{2},P_{3})$ one separates the points $x_{i}$ in $P_{i}$ into an $\alpha$-sink $x_{i}^{-}$ and an $\alpha$-source $x_{i}^{+}$ and one adds 
arrows 
$\gamma_{i}: x_{i}^{-} \rightarrow x_{i}^{+}$ replacing the loops. Then one identifies $x_{i}^{+}$ with $x_{i+1}^{-}$ where the indices are taken modulo $3$ so that one obtains a cycle of 
length $3$ containing the $\gamma$-arrows. This is the only triangle such that the quiver remains connected after removing the arrows of the triangle ( and no point ). Now for  $e<m(T)$ the algebra 
$D(P_{1},P_{2},P_{3})$ has an additional symmetry $\sigma$ induced by $\tau^{e}$. This 
is only possible for $P_{1}=P_{2}=P_{3}=P$, $n= 3m$, $e=2m-1$ and for the rotation of the $\gamma$-cycle as additional symmetry.

 We do not give in detail the relations that $D(P_{1},P_{2},P_{3})$ has to satisfy, but we remark that $\sigma$ respects the relations on $D(P,P,P)$ and we denote by $N(P)$ the quotient.
 Then the quiver of $N(P)$ is obtained from $P=P_{1}$ by adding a loop $\gamma$ in the point $x_{1}$ 
and it follows from the relations of $D(P,P,P)$ that $N(P)$ has a penny-farthing with loop $\gamma$ as the only non-deep contour and the $\alpha$-path of length 2 through $x_{1}$ as the
 only critical
path. From this one gets by an easy direct argument that in characteristic 2 the linearization of the stem-category of $N(P)$ is the only non-standard algebra in accordance with
\cite{Waschbsch1,Riedtmann2}.

 For the smallest
possible Brauer-quiver algebra with two points one finds as $D(P,P,P)$ the algebra whose quiver is shown in figure 7.\vspace{0.5cm}

\setlength{\unitlength}{0.8cm}
\begin{picture}(10,5)

\put(3.5,5){\circle*{0.1}}
\put(6.5,5){\circle*{0.1}}
\put(9.5,5){\circle*{0.1}}
\put(5,3){\circle*{0.1}}
\put(8,3){\circle*{0.1}}
\put(6.5,1){\circle*{0.1}}

\put(6.5,5){\vector(-1,0){3}}
\put(9.5,5){\vector(-1,0){3}}
\put(8,3){\vector(-1,0){3}}
\put(3.5,5){\vector(3,-4){1.5}}
\put(5,3){\vector(3,-4){1.5}}
\put(6.5,1){\vector(3,4){1.5}}
\put(8,3){\vector(3,4){1.5}}
\put(5,3){\vector(3,4){1.5}}
\put(6.5,5){\vector(3,-4){1.5}}

\put(6,0){figure 7}
\end{picture}

\vspace{0.3cm}  It admits the threefold rotation as a symmetry
 and the quotient $N(P)$ is the algebra $A(1)$ we started with in figure 1.

Our round-trip through the land of mild categories has come to an end. We conclude with the proof that the $N(P)$ from above are the only non-standard algebras.
\begin{theorem}
Let $A\neq k$ be a connected representation-finite selfinjective algebra with standard-form $A^s$. Then we are in one of the following two situations:

\begin{enumerate}
 \item  $A^{s}$ is not isomorphic to some $N(P)$ for a Brauer-quiver $P$ with at most 2 vertices. Then all arrows are bitransit and there is no non-deep contour. Consequently $A$ is standard.
 
\item $A^{s}$ is  isomorphic to some $N(P)$ for a Brauer-quiver $P$ with at most 2 vertices. Then there is exactly one critical path and one non-deep contour.
$A$ is isomorphic to the linearization of its ray-category or of its stem-category and these two are isomorphic iff the characteristic is not $2$.
\end{enumerate}
\end{theorem}

\begin{proof} Part b) follows from the above discussion. 
For part a) we work with ray-categories. So let $C$ be a configuration for the tree class $A_{n}$ with period $e$ and $ef=n$: Denote by $P_{s}$ the ray category of the full subcategory of 
 projectives in $k({\bf Z}T_{C}/\langle \tau^{se} \rangle )$. Then $P_{f}$ is the ray category of a Brauer-quiver algebra and we have a $f$-fold Galois covering $\pi:P_{f} \rightarrow P_{1}$.
We claim that all irreducible morphisms - whence all morphisms - in $P_{1}$ are bitransit and that $P_{1}$ has no non-deep contour. Now up to duality any such contour involves three
 irreducible morphisms $\gamma,\delta,\epsilon$ with $\gamma \neq \delta$ and 
$\epsilon\delta \neq 0$ and $\epsilon\gamma \neq 0$  ( see figure 4 ). This situation can be lifted to $P_{f}$ where it is impossible.

 Next let $\gamma:x \rightarrow y$ be an irreducible
 morphism in $P_{1}$ with $x\neq y$ and let $\rho= \rho_{m}\ldots \rho_{2}\rho_{1}$ be a generator of the 'radical' of $P_{1}(y,y)$ written as a product of irreducible morphisms.
Assume that $\rho\gamma \neq 0$. Lifting $x$ and the irreducible morphisms one obtains upstairs a non-zero product $\sigma_{m+1}\sigma_{m}\ldots \sigma_{1}$ with
$\pi(\sigma_{1})=\gamma$. Thus all $\sigma_{i}$ belong to a  cycle  of the Brauer-quiver say with $r$ points and one has $\sigma_{r}\ldots \sigma_{m}\ldots \sigma_{1}\neq 0$ for some 
 irreducible morphisms $\sigma_{j}$ and also $\sigma_{m}\ldots \sigma_{1}\sigma_{r}\ldots \sigma_{m+1}\neq 0$. We have $\pi(\sigma_{1}\sigma_{r}\ldots \sigma_{m+2})=\rho^{i}$ with $i\geq 1$ whence 
$\pi(\sigma_{1})$ and $\pi(\sigma_{m+1})$ are two irreducible morphisms not annihilated by $\rho$. So they coincide as seen above. This gives $\rho\gamma=\pi( \sigma_{m+1}\sigma_{m}\ldots 
\sigma_{1})=\gamma\pi(\sigma_{m}\ldots \sigma_{1})$ and $\gamma$ is cotransit. Dually one gets that $\gamma$ is transit.

Now for any $s$ we have a Galois covering $P_{s} \rightarrow P_{1}$. Any non-bitransit morphism in $P_{s}$ would produce one in $P_{1}$ which is impossible.
 For non-deep contours the same argument works.

Finally, let $C$ be a configuration for the tree class $D_{3m}$ with period $e=2m-1$. Using  conventions analogous to the above we know that $P_{1}$ is $N(P)$ with only one non-deep 
contour which is a penny-farthing with loop in $x$ and only two non-bitransit irreducible morphisms $\alpha_{1}$ starting in $x$ and $\alpha_{n}$ ending in $x$.
 We only have to look at the $2$-fold Galois covering $\pi:P_{2} \rightarrow P_{1}$ that preserves non-bitransit irreducible morphisms and non-deep contours. One checks easily that
 the preimages of the $\alpha_{i}$ are bitransit and that there
 is no non-deep contour in $P_{2}$ that is mapped to the penny-farthing.
 
\end{proof}


\begin{thebibliography}{99}
 \bibitem{Auslander}M.Auslander, I.Reiten, S.O.Smal{\o}: Representation theory of artin algebras, Cambridge studies in advanced mathematics 36.
\bibitem{Larrion}R. Bautista, F.Larri\'{o}n: Auslander Reiten quivers for certain algebras of finite representation type,J. London Math. Soc. 26, ( 1982), 43-52.

\bibitem{BGRS}R.Bautista, P.Gabriel, A.V.Roiter, L.Salmer\'{o}n: Representation-finite algebras and multiplicative bases, Invent. Math. 81(1985),  217-285.
\bibitem{Bautista}R.Bautista: On algebras of strongly unbounded representation type, Comment.Math.Helv. 60 (1985), 392-399.
\bibitem{Zyk}K.Bongartz: Zykellose Algebren sind nicht zügellos. Representation theory, II (Proc. Second Internat. Conf., Carleton Univ., Ottawa, Ont., 1979), pp. 97–102,
 Lecture Notes in Math., 832, Springer, Berlin, 1980. 
\bibitem{Treue}K.Bongartz: Treue einfach zusammenhängende Algebren I, Comment.  Math. Helv.  57(1982), 282-330.
\bibitem{quadratic}K.Bongartz: Algebras and quadratic forms. J. London Math. Soc. (2) 28 (1983), no. 3, 461–469.
\bibitem{Criterion}K.Bongartz: A criterion for finite representation type, Math. Ann. 269(1984), 1-12.
\bibitem{liste}K.Bongartz: Critical simply connected algebras, Manuscr. Math. 46(1984), 117-136.
\bibitem{Extensions}K.Bongartz: Indecomposables over representation-finite algebras are extensions of an indecomposable and a simple,  Math. Z. 187(1984),  75-80.
\bibitem{Standard}K.Bongartz: Indecomposables are standard, Comment. Math.   Helv. 60(1985), 400-410.
\bibitem{Degenerations}K.Bongartz: On degenerations and extensions of finite dimensional modules, Adv.  Math.121(1996), 245-287.
\bibitem{Gaps}K.Bongartz: Indecomposables live in all smaller lengths, to appear in Representation Theory.
\bibitem{mild}K.Bongartz: On mild contours in ray-categories, to appear in Algebras and representation theory.
\bibitem{Coverings}K.Bongartz, P.Gabriel: Covering  spaces in representation theory, Invent. Math. 65(1982), 331-378.

\bibitem{BLR}O.Bretscher, C.Läser, C.Riedtmann: Self-injective and simply connected algebras. Manuscripta Math. 36 (1981/82), no. 3, 253–307.
\bibitem{Bretscher}O.Bretscher, P.Gabriel:The standard form of a representation-finite algebra, Bull. Soc. Math. France 111, (1983), 21-40.
\bibitem{Todorov}O.Bretscher, G.Todorov: On a theorem of Nazarova and Roiter, Proc. ICRA IV, Lecture Notes 1177, ( 1986), 50-54.

\bibitem{Brstle} T.Brüstle, J.A.de la Peña, A.Skowroński: Tame algebras and Tits quadratic forms. Adv. Math. 226 (2011), no. 1, 887–951.
 
\bibitem{CB} W.Crawley-Boevey: Tameness of biserial algebras. Arch. Math. (Basel) 65 (1995), no. 5, 399–407. 
\bibitem{CBD}W.Crawley-Boevey: On tame algebras and bocses, Proc. London Math. Soc. 56(1988), 451-483.
\bibitem{Dade} E.C.Dade: Blocks with cyclic defect groups. Ann. of Math. (2) 84 1966 20–48.
\bibitem{delapena1} J.A.de la Peña, U.Fischbacher: Algorithms in representation theory of algebras.
 Representation theory, I (Ottawa, Ont., 1984), 115–134, Lecture Notes in Math., 1177, Springer, Berlin, 1986.
\bibitem{delapena2}J.A.de la Peña, R.Martínez-Villa: The universal covering of a quiver with relations. J. Pure Appl. Algebra 30 (1983), no. 3, 277–292. 
\bibitem{Dowbor}P.Dowbor, A.Skowroński: Galois coverings of representation-infinite algebras. Comment. Math. Helv. 62 (1987), no. 2, 311–337. 
\bibitem{Nor} P.Dräxler: Normal forms for representations of representation-finite algebras. J. Symbolic Comput. 32 (2001), no. 5, 491–497. 
\bibitem{Drxler}P.Dräxler: Aufrichtige gerichtete Ausnahmealgebren. (German) [Exceptional sincere directed algebras] Bayreuth. Math. Schr. No. 29 (1989).
\bibitem{Drozd}Yu.A.Drozd: On tame and wild matrix problems, in: Matrix problems, Akad.Nauk.Ukr. S.S.R., Inst. Matem.,Kiev 1977, 104-114. ( in Russian ).
 \bibitem{fIbt} U.Fischbacher: Une nouvelle preuve d'un théorème de Nazarova et Roiter. (French) [A new proof of a theorem of Nazarova and Roiter] 
C. R. Acad. Sci. Paris Sér. I Math. 300 (1985), no. 9, 259–262. 
\bibitem{Fischbacher}U.Fischbacher: Zur Kombinatorik der Algebren mit endlich vielen Idealen, J. Reine Angew. Math. 370 (1986), 192-213.
\bibitem{Fisch3}U.Fischbacher: The representation-finite algebras with at most 3 simple modules. Representation theory I (Ottawa, Ont., 1984), 
94–114, Lecture Notes in Math., 1177, Springer, Berlin, 1986. 
\bibitem{Dynkingab}P.Gabriel: Unzerlegbare Darstellungen I. (German) Manuscripta Math. 6 (1972), 71–103; correction, ibid. 6 (1972), 309.
\bibitem{Indecom}P.Gabriel: Indecomposable representations II, Symp. Math. 11(1973), 81-104.
\bibitem{Open}P.Gabriel: Finite representation type is open, Springer lecture notes 488, 1975, 132-155.

\bibitem{Groups}P.Gabriel, C.Riedtmann: Group representations without groups. Comment. Math. Helv. 54 (1979), no. 2, 240–287. 
\bibitem{Gabaus} P.Gabriel: Auslander-Reiten sequences and representation-finite algebras, Representation theory I, (Proc. Workshop, Carleton Univ., Ottawa, Ont., 1979),
 pp. 1–71, Lecture Notes in Math., 831, Springer, Berlin, 1980.
\bibitem{Galoiscovering}P.Gabriel: The universal cover of a representation-finite algebra,  Lect. Not. Math. 903(1981), 68-105.
\bibitem{Revgab}P.Gabriel: Darstellungen endlichdimensionaler Algebren. (German) [Representations of finite-dimensional algebras] Proceedings of the
 International Congress of Mathematicians, Vol. 1, 2 (Berkeley, Calif., 1986), 378–388, Amer. Math. Soc., Providence, RI, 1987.
\bibitem{Buch}P.Gabriel, A.V.Roiter: Representations of finite-dimensional algebras, Vol.73 of the encyclopaedia of Math. Sciences (1992), 1-177.
\bibitem{Zahmwild}P.Gabriel, L.A.Nazarova, A.V.Roiter, V.V.Sergeichuk, D.Vossieck: Tame and wild subspace problems, Ukrainian J.Math. 45 (1993), 335-372.
\bibitem{Geissm}C.Geiss: Darstellungsendliche Algebren und multiplikative Basen, Diplomarbeit Uni Bayreuth 1990, 126 pages.
\bibitem{Geiss} C.Geiss: On degenerations of tame and wild algebras, Arch. Math. (Basel) 64 (1995), no. 1, 11–16. 
\bibitem{GP} I.M.Gelfand, V.A.Ponomarev: Indecomposable representations of the Lorentz group. (Russian) Uspehi Mat. Nauk 23 1968 no. 2 (140), 3–60.
\bibitem{Gordon}R.Gordon,E.L.Green: Representation theory of graded Artin algebras. J. Algebra 76 (1982), no. 1, 138–152. 
\bibitem{Green}R.Gordon,E.L.Green: Graded Artin algebras. J. Algebra 76 (1982), no. 1, 111–137.
\bibitem{Green2}E.L.Green: Graphs with relations, coverings and group-graded algebras. Trans. Amer. Math. Soc. 279 (1983), no. 1, 297–310.
\bibitem{Listebild}D.Happel, D.Vossieck: Minimal algebras of infinite representation type with preprojective component, Manuscr. Math. 42(1983), 221-243.
\bibitem{Higman}D.G.Higman: Indecomposable representations at characteristic p, Duke Math. J. 21, (1954). 377–381.
\bibitem{Hughes}D.Hughes, J.Waschbüsch: Trivial extensions of tilted algebras. Proc. London Math. Soc. (3) 46 (1983), no. 2, 347–364.
\bibitem{Jans}J.P.Jans: On the indecomposable representations of algebras. Ann. of Math. (2) 66 (1957), 418–429.
\bibitem{Janusz}G.J.Janusz: Indecomposable modules for finite groups. Ann. of Math. (2) 89 1969 209–241.
\bibitem{JensenLenzing}C.Herrmann, C.U.Jensen, H.Lenzing: Applications of model theory to representations of finite-dimensional algebras. Math. Z. 178 (1981), no. 1, 83–98. 
\bibitem{Kasjan} S.Kasjan: Representation types of algebras from the model theory point of view. Trends in representation theory of algebras and related topics, 433–465, EMS Ser. Congr. Rep., Eur. Math. Soc., Zürich, 2008.
\bibitem{Kleiner}M.M.Kleiner: Partially ordered sets of finite type. (Russian) Investigations on the theory of
 representations. Zap. Naučn. Sem. Leningrad. Otdel. Mat. Inst. Steklov. (LOMI) 28 (1972), 32–41.
\bibitem{Kleiner2}M.M.Kleiner: Faithful representations of partially ordered sets of finite type. (Russian) Investigations on the theory of representations. Zap. Naučn. Sem. Leningrad. Otdel. Mat. Inst. Steklov. (LOMI) 28 (1972), 42–59. 
\bibitem{Kupisch1}H.Kupisch: Symmetrische Algebren mit endlich vielen unzerlegbaren Darstellungen. I. (German) J. Reine Angew. Math. 219 1965 1–25. 
\bibitem{Kupisch}H.Kupisch: Unzerlegbare Moduln endlicher Gruppen mit zyklischer p-Sylow-Gruppe. (German) Math. Z. 108 1969 77–104.
\bibitem{Kupisch2}H.Kupisch: Symmetrische Algebren mit endlich vielen unzerlegbaren Darstellungen. II. (German) J. Reine Angew. Math. 245 1970 1–14.
\bibitem{Kupisch4}H.Kupisch: Basisalgebren symmetrischer Algebren und eine Vermutung von Gabriel. (German) J. Algebra 55 (1978), no. 1, 58–73. 
 Correction in J. Algebra 74 (1982), no. 2, 543–552.
\bibitem{Kupisch3}H.Kupisch, E.Scherzler: Symmetric algebras of finite representation type. Representation theory, II (Proc. Second Internat. Conf., Carleton Univ.,
 Ottawa, Ont., 1979), pp. 328–368, Lecture Notes in Math., 832, Springer, Berlin, 1980. 
\bibitem{Kupisch5}H.Kupisch, J.Waschbüsch: On multiplicative bases in quasi-Frobenius algebras. Math. Z. 186 (1984), no. 3, 401–405.
\bibitem{Martinez}R.Martinez-Villa: Algebras stably equivalent to factors of hereditary. Representations of algebras (Puebla, 1980),
 pp. 222–241, Lecture Notes in Math., 903, Springer, Berlin-New York, 1981. 
\bibitem{Nazarova} L.A.Nazarova,A.V.Roiter: Representations of partially ordered sets. (Russian) Investigations on the theory of representations.
 Zap. Naučn. Sem. Leningrad. Otdel. Mat. Inst. Steklov. (LOMI) 28 (1972), 5–31. 
\bibitem{NR} Nazarova, L. A.; Roiter, A. V. Kategorielle Matrizen-Probleme und die Brauer-Thrall-Vermutung. (German) Aus dem Russischen übersetzt von K. Nikander. 
Mitt. Math. Sem. Giessen Heft 115 (1975), i+153 pp.
\bibitem{Ovsienko}S.A.Ovsienko: Integral weakly positive forms, in Schur matrix problems and quadratic forms, Inst.Mat.Akad.Nauk.Ukr. S.S.R.,Inst.Mat.Kiev, 1979,pp- 106-126, ( in Russian ).

\bibitem{Riedtmann}C.Riedtmann: Algebren, Darstellungsköcher, Überlagerungen und zurück, Comment. Math. Helv. 55(1980), 199-224.
\bibitem{RiedtmannAn}C.Riedtmann:  Representation-finite selfinjective algebras of class $A_{n}$. Representation theory, II 
(Proc. Second Internat. Conf., Carleton Univ., Ottawa, Ont., 1979), pp. 449–520, Lecture Notes in Math., 832, Springer, Berlin, 1980. 
\bibitem{Revriedt}C.Riedtmann: Algèbres de type de représentation fini (d'après Bautista, Bongartz, Gabriel, Roiter et d'autres). (French) 
[Algebras of finite representation type (following Bautista, Bongartz, Gabriel, Roiter et al.)] Seminar Bourbaki,
 Vol. 1984/85. Astérisque No. 133-134 (1986), 335–350. 
\bibitem{Riedtmann2}C.Riedtmann:  Representation-finite selfinjective algebras of class $D_{n}$. Compositio Math. 49 (1983), no. 2, 231–282.
\bibitem{Riedt}C.Riedtmann: Configurations of $ZD_{n}$. J. Algebra 82 (1983), no. 2, 309–327. 
\bibitem{Ringel}C.M.Ringel: On algorithms for solving vector space problems. II. Tame algebras. Representation theory, I (Proc. Workshop, Carleton Univ., Ottawa, Ont., 1979), pp. 137–287, Lecture Notes in Math., 831, Springer, Berlin, 1980. 
\bibitem{Tame}C.M.Ringel: Tame algebras and integral quadratic forms, Lect. Not. Math. 1099(1984), 1-376.
\bibitem{Ringel2}C.M.Ringel: Indecomposables live in all smaller lengths. Bull. Lond. Math. Soc. 43 (2011), no. 4, 655–660.
\bibitem{Rinmin}C.M.Ringel: The minimal representation-infinite algebras which are special biserial. Representations of algebras and related topics, 501–560, EMS Ser. Congr. Rep., Eur. Math. Soc., Zürich, 2011.
\bibitem{RogatTesche}A.Rogat, T.Tesche: The Gabriel quivers of the sincere simply connected algebras, Diplomarbeit Universität Wuppertal 1992, 69 pages,
\bibitem{BTI}A.V.Roiter: The unboundedness of the dimensions of the indecomposable representations of algebras that have an infinite number of indecomposable representations,
 Izv. Akad. Nauk SSSR,Ser. Mat. 32(1968), 1275-1282,English transl.: Math. USSR, Izv. 2(1968), 1223-1230.
\bibitem{Bongo}A.V.Roiter:  Generalization of Bongartz theorem, Preprint  Math.  Inst. Ukrainian Acad. of Sciences, Kiev(1981),1-32.
\bibitem{Skowronskisimson} D.Simson, A.Skowroński: Elements of the representation theory of associative algebras. Vol. 2. Tubes and concealed algebras of Euclidean type.
 London Mathematical Society Student Texts, 71. Cambridge University Press, Cambridge, 2007. xii+308 pp.
\bibitem{Todo}G.Todorov: Almost split sequences for TrD-periodic modules, Proc.ICRA II, Ottawa 1979, Springer Lect.Not. 832, 579-599.
\bibitem{vonHhne} H.J.von Höhne: On weakly positive unit forms. Comment. Math. Helv. 63 (1988), no. 2, 312–336. 
\bibitem{Waschbsch1}J.Waschbüsch: Symmetrische Algebren vom endlichen Modultyp. (German) J. Reine Angew. Math. 321 (1981), 78–98.
\bibitem{Waschbsch}J.Waschbüsch: Universal coverings of selfinjective algebras. Representations of algebras (Puebla, 1980), pp. 331–349, Lecture Notes in Math., 903, Springer, Berlin-New York, 1981.
\bibitem{Wald}B.Wald, J.Waschbüsch:Tame biserial algebras. J. Algebra 95 (1985), no. 2, 480–500.
\end{thebibliography}
\end{document}